\documentclass[12pt]{amsart}
\usepackage[margin=30mm]{geometry}

\usepackage[utf8]{inputenc}
\usepackage{amsmath,amstext,amssymb,amsopn,amsthm}
\usepackage{url,verbatim}
\usepackage{mathtools}
\usepackage{enumerate}

\numberwithin{equation}{section}
\allowdisplaybreaks

\renewcommand{\P}{\mathbb{P}}
\newcommand{\D}{\mathbb{D}}
\renewcommand{\S}{\mathbf{S}}
\newcommand{\Sph}{\mathbb{S}}
\newcommand{\bF}{\mathbf{F}}
\newcommand{\cB}{\mathbf{B}}

\newcommand{\M}{\mathbb{M}}
\newcommand{\R}{\mathbb{R}}

\newcommand{\E}{\mathbb{E}}
\newcommand{\1}{\mathbf{1}}

\newcommand{\bv}{\mathbf{v}}
\newcommand{\ol}{\overline}
\newcommand{\sig}{y}
\newcommand{\eps}{\varepsilon}

\newtheorem{thm}{Theorem}[section]
\newtheorem{lemma}[thm]{Lemma}
\newtheorem{corol}[thm]{Corollary}
\newtheorem{propo}[thm]{Proposition}
\theoremstyle{definition}

\newtheorem{defin}[thm]{Definition}

\newtheorem{remark}[thm]{Remark}

\title{Random reflections in a high dimensional tube}
\author{Krzysztof Burdzy and Tvrtko Tadi\'c}

\address{KB: Department of Mathematics, Box 354350, University of Washington, Seattle, WA 98195, USA}
\email{burdzy@math.washington.edu}

\address{TT: Department of Mathematics, University of Zagreb, Bijenička cesta 30,
10000 Zagreb, Croatia \and Microsoft Corporation (City Center Plaza Bellevue), One Microsoft Way, Redmond, WA 98052}
\email{tvrtko@math.hr}

\thanks{KB: Research supported in part by NSF Grant DMS-1206276. TT: Research supported in part by Croatian Science Foundation grant 3526.}

\keywords{random reflections, stopped random walks, tail of the arccosine distribution,  undershoot, overshoot}
\subjclass[2010]{60G50, 60K05, 37D50, 37H99}
\begin{document}

\maketitle

\begin{abstract}
We consider light ray reflections in $d$-dimensional semi-infinite tube, for $d\geq 3$, made of Lambertian material.
The source of light is placed far away from the exit, and the light ray is assumed to reflect so that the distribution of the direction of the reflected light ray has the density proportional to the cosine of the angle with the normal vector. We present new results on the exit distribution from the tube,
and generalizations of some theorems from an earlier article, where the dimension was limited to $d=2$ and $3$.

\end{abstract}
\section{Introduction}

This paper is a continuation of \cite{laser}, where Lambertian reflections in a semi-infinite tube were introduced and studied. A source of light was placed far away from the tube opening and the light reflected from the walls of the tube according to the Lambertian law (also known as the Knudsen law in the theory of gases). The main results were concerned with the properties of the light ray at the exit time. The paper \cite{laser} was mostly concerned with the 2-dimensional case although it contained some results in the 3-dimensional case. 
This paper is devoted to the $d$-dimensional case, for any $d\geq 3$. The  exit distribution is dramatically different in two dimensions from three dimensions. We will show that the exit distributions are similar in the qualitative sense for all $d\geq 3$.

On the technical side, we will derive some new results about overshoot of random walks, arccosine distribution and products of random variables having this distribution, and distributions with regularly varying tails.

One of our main results is Theorem \ref{thm:uniformFromAway} which states that if the last position of the ray was at a large distance 
$\beta$ form the exit of the tube, the exiting point of the ray is (approximately) uniformly distributed on the exiting disc.
\begin{figure}\label{j26.2}
\begin{center}
 \includegraphics[width=8cm]{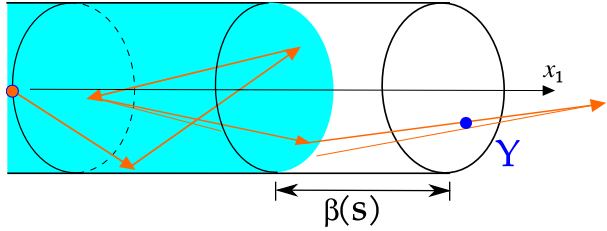}
\caption{Illustration of  Theorem \ref{thm:uniformFromAway}.}\label{sl3blue}
\end{center}
\end{figure}
This is illustrated in Figure \ref{sl3blue}. If we apply a filter on the exiting disc so that only the rays coming from the blue part of the tube 
are shown, the light on the exiting disc will be uniformly distributed. Results of a simulation 
illustrated in Figure \ref{simulation} show that this is practically true when  $\beta=3$ and the source of light is at the distance $50$ from the exiting disc.

Another of our main results is Theorem \ref{thm:u20_1}, the generalization of the brightness singularity result from \cite{laser}. 
Suppose that the light ray starts $s$ units from the opening of the tube.
Let $\bv_s  $ be the unit
vector representing the direction of the light ray at the exit time 
assuming it leaves the tube at the center of the opening.
Let $\cB(r) = \{(x_1,\ldots, x_d): x_1^2 + x_2^2 +\ldots+ x_d^2 = 1, x_2^2+\ldots +x_d^2 \leq r^2 , x>0 \}$ denote a ball on the unit sphere.
A somewhat informal statement of Theorem \ref{thm:u20_1} is
\begin{thm}
For some constant $c>0$ and any $0 < r_1 < r_2 < 1$,
\begin{align*}
\lim_{s\to \infty}
s^{d-2} \P\left(\bv_s \in \cB\left( \frac{r_2}{s}\right) \setminus \cB\left( \frac{r_1}{s}\right)\right) 
=    c(r^{d-2}_2-r^{d-2}_1).
\end{align*}
\end{thm}
Consider an observer at the center of the opening of the tube, looking towards the interior of the tube.
The theorem  says that small annuli at the center of the field of vision,
with the area of magnitude $1/s^{d-1}$, receive about $1/s^{d-2}$ units of light. Hence, the apparent brightness is about $s$ at the distance $1/s$ from the center, if the light source is $s$ units away from the opening. 
This means that the surface of the tube does not appear to be Lambertian, i.e., the surface does not have uniform apparent brightness. This can be explained by the fact that not all parts of the surface of the tube  receive the  same amount of light.

A brief review of  literature on random  reflections is given
in \cite{laser}; we only mention here some of the papers in this area: 
\cite{ABS, comets1,comets2,comets3, evans}.

The paper is organized as follows. We present some results on random walks 
and arccosine distribution in Section \ref{sec:prel}.
Section \ref{sec:description} contains a precise definition of Lambertian direction for a light ray and contains estimates of the exit distribution of the light ray from the tube in case there were no reflections inside the tube. The construction of the process of reflected light ray is given in 
Section \ref{sec:construction}. In the next section, Section \ref{sec:tail}, we included some precise results on the tail of the distribution of a single   flight of light ray between reflections. Section \ref{sec:exit} contains results on the exit distribution in the case when the last reflection is far from the exit, and the last section, Section \ref{sec:bright}, contains the rigorous statement and the proof of the brightness singularity result alluded to above. 

\section{Preliminaries}\label{sec:prel}

\subsection{Stopped random walks}
We will study a random walk $\{S_n, n\geq 0\}$, with $S_0=0$ and $S_n=S_{n-1}+X_{n}$ for $n\geq 1$,
where $\{X_n, n\geq 1\}$ is an i.i.d. sequence.

For $s>0$ we let
\begin{align}\label{eq:NsMk} 
N_s &= \inf\{n>0: S_n>s\},\\
O_s&=S_{N_s}-s, \qquad U_s=s-S_{N_{s}-1}.\nonumber
\end{align}
We call $O_s$ the overshoot and $U_s$ the undershoot 
of 
the random walk $S_n$ at $s$.

\begin{defin}
For a function $h:\R^+\to \R^+$ we say that it is regularly varying with exponent (index)
$\alpha$ if 
\begin{equation}
\lim_{t\to\infty}h(xt)/h(t)=x^{\alpha} \label{limit1} 
\end{equation}
 for $x>0$. A function $h$ is 
called slowly varying if $\alpha =0$.
\end{defin}

It is well known that $h$ is a 
regularly varying function  with index $\alpha$ if and only if it is of the form $h(x) = x^\alpha L(x)$ where  $L$ is a slowly varying function.

We will use the notation $f\sim g$ to indicate that
$\lim_{x\to a} f(x)/g(x)=1$, where $a$ will depend on the context.

The following  theorem can be found in  \cite[Thm.~1.5.2]{regularVariation}.

\begin{thm}\label{thm:limunf}
Suppose that $h:\R^+\to \R^+$ is regularly varying with index $\alpha<0$. Then for every $a>0$,
the limit in $(\ref{limit1})$ is uniform in $x\in[a,\infty)$.
\end{thm}

\begin{thm}\label{thm:overUnderShoot}
Let $S_0=0$, $S_n=X_1+\ldots+X_n$, where $(X_n)$ is a sequence of continuous i.i.d. random variables
such that $x\mapsto \P(X_1>x)$ is regularly varying with index $-\alpha<0$. Then for $\varepsilon\in (0,1)$ and $t\in [0,1]$,
\begin{equation}\label{general:limit:result}
 \lim_{s\to\infty}\P\left(\frac{U_s}{U_s+O_s}\leq t \mid S_{N_s-1}\leq s(1-\varepsilon)\right)=t^{\alpha}.
\end{equation}

If $\P(X_1<0)>0$ then $(\ref{general:limit:result})$ is true for all $\varepsilon >0$.
\end{thm}
\remark Statement $(\ref{general:limit:result})$ is written in the form that we will use later. It can be reformulated as follows. For $x\geq 1$,
$$\lim_{s\to\infty} \P\left(\frac{X_{N_s}}{U_s}>x \mid U_s\geq \varepsilon s\right)=x^{-\alpha}.$$
\begin{proof}[Proof of Theorem \ref{thm:overUnderShoot}]
We need to check whether $\P(S_{N_s-1}\leq s(1-\varepsilon))> 0$ for $s>0$ so that $(\ref{general:limit:result})$ makes sense.
For $\varepsilon\in(0,1)$, we have 
\begin{align*}
\P(S_{N_s-1}\leq s(1-\varepsilon))
\geq \P(S_0\leq s(1-\varepsilon),S_1>s)=\P(X_1>s)>0.
\end{align*}
If $\P(X_1<0)>0$, then there exist $\delta_1>\delta_2 >0$ such that $\beta := \P(-\delta_1< X_1<-\delta_2)>0$. For any $\varepsilon \geq 1$, there exists 
$m$ such that $m\delta_2 > s(\varepsilon-1)$. Then 
\begin{align*}
 \P(S_{N_s-1}\leq s(1-\varepsilon))&\geq \P(S_{m+1}>s, S_m\leq s(1-\varepsilon),S_{m-1}<s, \ldots S_1<s)\\
& \geq \P(-\delta_1< X_1<-\delta_2,\ldots, -\delta_1< X_m<-\delta_2, X_{m+1}>s+m\delta_1)\\
&=\beta^m \P(X_{1}>s+m\delta_1)>0. 
\end{align*}

By definition we have,
\begin{align*}
\P&\left(\frac{U_s}{U_s+O_s}\leq t, S_{N_s-1}\leq s(1-\varepsilon)\right)\\
&=\P\left(tX_{N_s}\geq (s-S_{N_s-1}), S_{N_s-1}\leq s(1-\varepsilon)\right)\\
&=\sum_{k=1}^{\infty}\P\left(tX_{k}\geq (s-S_{k-1}), S_{k-1}\leq s(1-\varepsilon), N_s=k\right)\\
&=\sum_{k=1}^{\infty}\P\left(X_{k}\geq (s-S_{k-1})/t, S_{k-1}\leq s(1-\varepsilon), S_{k-2}<s,\ldots, S_1<s\right)\\
&=\sum_{k=1}^{\infty}\int_{-\infty}^{s(1-\varepsilon)}\P\left(X_{k}\geq (s-y)/t\right)\P( S_{k-1}\in dy, S_{k-2}<s,\ldots, S_1<s)\\
&& (y=su)\\
&=\sum_{k=1}^{\infty}\int_{-\infty}^{1-\varepsilon}\P\left(X_{1}\geq (s-su)/t\right)\P( S_{k-1}\in s\, du, S_{k-2}<s,\ldots, S_1<s).
\end{align*}
By Theorem \ref{thm:limunf}, since $1-u\geq \varepsilon>0$, we have   $$\frac{\P\left(X_{1}\geq (s(1-u))/t\right)}{\P\left(X_{1}\geq s(1-u)\right)}\to t^{\alpha}
$$ 
as $s\to \infty$,
uniformly in $t \in[0,1]$. Hence, as $s\to \infty$,
\begin{align*}
\P&\left(\frac{U_s}{U_s+O_s}\leq t, S_{N_s-1}\leq s(1-\varepsilon)\right)\\
&\sim \sum_{k=1}^{\infty}\int_{-\infty}^{1-\varepsilon}t^\alpha\P\left(X_{1}\geq s-su\right)\P( S_{k-1}\in s\, du, S_{k-2}<s,\ldots, S_1<s)\\
&=t^\alpha\sum_{k=1}^{\infty}\int_{-\infty}^{s(1-\varepsilon)}\P\left(X_{k}\geq s-su\right)\P( S_{k-1}\in dy, S_{k-2}<s,\ldots, S_1<s)\\
&=t^\alpha\sum_{k=1}^{\infty}\P\left(X_{k}\geq s-S_{k-1}, S_{k-1}\leq s(1-\varepsilon), N_s=k\right)\\
&=t^\alpha \P(S_{N_s-1}\leq s(1-\varepsilon)).
\end{align*}
\end{proof}

\subsection{Arccosine distribution}\rm The arcsine distribution is well known for its many interesting properties, see, for example \cite{arcsine}.
In this section we will study a similar distribution which we will call the arccosine distribution. We will need this material in Section \ref{sec:description} because we will study
 products of random variables of the form $\cos \Phi_1\cdots \cos \Phi_n$
where $\Phi_1,\ldots, \Phi_n$ are i.i.d. random variables with distribution $U(-\pi/2,\pi/2)$ (uniform on $(-\pi/2,\pi/2)$).

Let
\begin{equation*}
 F=\cos \Phi, 
\end{equation*}
where $\Phi\stackrel{d}{=} U(-\pi/2,\pi/2)$. We will call the distribution ``arccosine'' and we will use $F$  to denote a random variable with this distribution in the rest of the paper.
Note that $F\in [0,1]$, a.s. 

\begin{lemma}\label{lema:tail1}We have for $x\in(0,1)$,
\begin{equation*}
 P(F>x)=\frac{2}{\pi}\arccos x. 
\end{equation*}
The density of $F$ is $\frac{2}{\pi \sqrt{1-x^2}}$
for $x\in (0,1)$ and $0$ otherwise. Furthermore, as $x\to 1^-$,
\begin{equation*}
 \P(F>x)\sim\frac{2\sqrt{2}}{\pi}\sqrt{1-x}.
\end{equation*}
\end{lemma}
We omit the proof because it is elementary.

Recall that $0!!=1$, $1!!=1$, and
 $n!!= n\cdot (n-2)!! $ for integers $n\geq 2$.

\begin{propo}\label{thm:prodarcsine} 
Let $(F_k)$ be an i.i.d. sequence of random variables distributed as $F$. For $n\geq 1$, we have $\P(F_1F_2\cdots F_n>x)\sim c_n (1-x)^{n/2}$ as $x\to 1^-$, where 
$c_1=\frac{2\sqrt{2}}{\pi}$, $c_2=\frac{2}{\pi}$ and for $n\geq 2$ we have
$$c_{n}=\frac{4}{\pi n}c_{n-2}.$$
Hence, 
\begin{equation*}
 \P(F_1F_2\cdots F_n>x)\sim 
\begin{cases}
\displaystyle
\frac{1}{n!!}\left(\frac{4}{\pi}\right)^{n/2}(1-x)^{n/2}, & \text{  if $n$ is even}, \\
\displaystyle
\frac{1}{\sqrt{2}\, n!!}\left(\frac{4}{\pi}\right)^{ (n+1)/2}(1-x)^{n/2}, & \text{  if $n$ is odd} .
\end{cases}
\end{equation*}
\end{propo}

\begin{proof}
The claim holds for $n=1$ by Lemma \ref{lema:tail1}.
Suppose $n=2$. Then
for $x\to 1^-$,
 \begin{align*}
  \P(F_1F_2>x) &= \int_x^1\P\left(F_1>\frac{x}{s}\right)\frac{2}{\pi\sqrt{1-s^2}}ds\\
	      & \sim 4\int_x^1\frac{1}{\pi^2}\sqrt{1-\frac{x}{s}}\frac{1}{\sqrt{1-s}}\frac{\sqrt{2}}{\sqrt{1+s}}ds\\ 
& \sim 4\int_x^1\frac{1}{\pi^2}\sqrt{1-\frac{x}{s}}\frac{1}{\sqrt{1-s}}ds\\
	      & = \frac{4}{\pi^2}\int_x^1\sqrt{\frac{s-x}{1-s}}ds =\frac{2(1-x)}{\pi}.
 \end{align*}
 We now consider $n\geq 3$ and assume that our claim holds for 
$n-1$ and $n-2$. The last asymptotic estimate yields for $x\to 1^-$,
\begin{align*}
  \P&(F_1F_2\cdots F_{n}>x) \\
 &=\int_x^1 \P\left(F_1F_2\cdots F_{n-2}>\frac{x}{s}\right)\P(F_{n-1}F_{n}\in ds)\\
 &\sim \int_x^1 \P\left(F_1F_2\cdots F_{n-2}>\frac{x}{s}\right)\frac{2}{\pi}ds\\
&\sim  \frac{2c_{n-2}}{\pi}\int_x^1 \left(1-\frac{x}{s}\right)^{(n-2)/2}ds\\
&\sim  \frac{2c_{n-2}}{\pi}\int_x^1 \left(s-x\right)^{(n-2)/2}ds=\frac{4c_{n-2}}{\pi n}(1-x)^{n/2}.
\end{align*}
The proposition now follows by induction.
\end{proof}

\section{Light ray in cylinder}\label{sec:description}\rm 

We will study a light reflection model in a semi-infinite $d$-dimensional 
cylinder, generalizing the setup introduced in \cite{laser}.
Consider a semi-infinite cylinder in $d$ dimensions
 $$C=\{(x_1,x_2,\ldots, x_d)\in \R^d\, :\, x_2^2+\ldots+x_d^2=1, x_1\leq 0\}.$$ 
The cross sections will be denoted 
$$\S_{s}=\{(s,x_2,\ldots, x_d)\in \R^d\, :\, x_2^2+\ldots+x_d^2=1\}.
$$
We will assume that the light ray reflects from the cylinder surface  and stays inside the cylinder until it exits C, i.e., it crosses $\S_0$
(see Figure \ref{sl3}. for the case $d=3$).
\begin{figure}
\begin{center}
 \includegraphics[width=9cm]{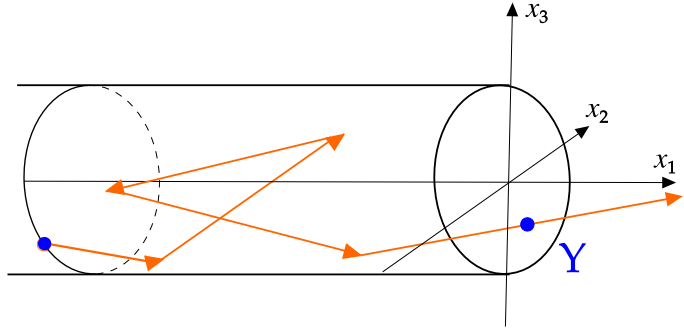}
\caption{Reflections in cylinder; 3-dimensional illustration. }\label{sl3}
\end{center}
\end{figure}
The reflections will be Lambertian---we formalize this as follows.
\begin{itemize}
 \item The ray starts at a point $S_0$ uniformly chosen in $\S_{-s}$
for some $s>0$. The initial velocity vector points into the interior of the cylinder and has the same distribution as the one at any reflection point, described below.
 \item  Whenever the ray hits $C$, it reflects. The angle $\Theta$ between the reflected ray and  the inner normal vector to the surface at the reflection point $Q$ (see Figure \ref{3drefl}) has the following density, 
\begin{figure}
\begin{center}
 \includegraphics[width=9cm]{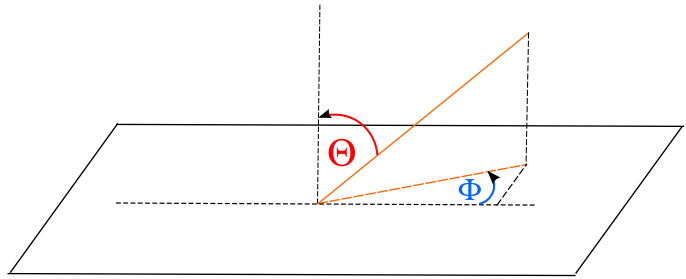}
\caption{Reflection vector encoding; 3-dimensional illustration.}\label{3drefl}
\end{center}
\end{figure}
\begin{equation}
 f(\theta) = \left\{\begin{array}{cr}
                         \frac{1}{2}\cos \theta, & \text{  for  }\theta \in (-\pi/2,\pi/2), \\
			  0,&  \textrm{otherwise}.
                        \end{array}
\right. \label{eq1}
\end{equation}
\item Given $\Theta$, the distribution of the point of intersection of the $(d-1)$-dimensional 
unit sphere centered at $Q$ and projection of the reflected ray on the tangent hyper plane is  uniformly 
distributed  on the intersection of the sphere and the tangent hyperplane
(see Figure \ref{3drefl}).
\end{itemize}
 
If $V$ has the uniform distribution $ U(-1,1)$ then it is easy to check that the following  equalities hold
in the sense of distribution,
\begin{align}\label{lem:theta}
\sin\Theta = V,\
 \cos\Theta = \sqrt{1-V^2},\
 \tan \Theta= \frac{V}{\sqrt{1-V^2}}.
\end{align}

Suppose that the light ray emanates from the point $(-u,0,\ldots,0,-1)$ 
for some $u>0$. 
The tangent plane to the cylinder $C$ at this point is $\{x:x_d=-1\}$ 
and the ray is moving in the direction of the random vector 
\begin{align}\label{eq:reflvector}
 \bv(R)=(&R\sin\Theta\cos \Phi_{d-2}\cos \Phi_{d-3}\dots\cos \Phi_{2}\cos \Phi_1,\\ 
&R\sin\Theta\cos \Phi_{d-2}\cos \Phi_{d-2}\dots\cos \Phi_2\sin \Phi_1,
\nonumber\\
&R\sin\Theta\cos \Phi_{d-2}\cos \Phi_{d-2}\dots\cos \Phi_3\sin \Phi_2,
\nonumber\\
&  \dots , \nonumber\\
&R\sin\Theta\cos \Phi_{d-2}\sin \Phi_{d-3},\nonumber\\
&  R\sin\Theta \sin \Phi_{d-2},\nonumber\\
 &R\cos\Theta),\nonumber
\end{align}
where $\Phi_1,\ldots \Phi_{d-2}$ are i.i.d. uniform $U(-\pi/2,\pi/2)$, 
and $R>0$.

Let the infinite cylinder be denoted
 $$\ol C=\{(x_1,x_2,\ldots, x_d)\in \R^d\, :\, x_2^2+\ldots+x_d^2=1\}.
$$ 

\begin{lemma}
Suppose that a light ray starts from a point $x=(-u,0,\ldots,0,-1)$ and 
moves in the direction $\bv(R)$ given by \eqref{eq:reflvector}. 
The ray will exit $\ol C$ at a point $y$ whose distance from $x$ is 
\begin{equation}
 R=\frac{2\cos\Theta}{1-\sin\Theta^2\cos^2 \Phi_{d-2}\ldots\cos^2 \Phi_1}.\label{eq:radius}
\end{equation}
\end{lemma}

\begin{proof}
We need to find $R>0$ such that $y= x + \bv(R)$ 
and $y_2^2+y_3^2+\dots+y_d^2=1$. A straightforward calculation yields
\eqref{eq:radius}.
\end{proof}

If a light ray starts at $(-u,0,\ldots,0,-1)$ and moves in the direction 
of the vector $(\ref{eq:reflvector})$ then it is easy to check that it 
intersects the plane $\{x:x_1=0\}$ at a point $Z_u=(0,\widetilde{Z}_u)$
given by
\begin{align}
\label{eq:planeIntersectionCoordinates}
\left(0, \frac{u\sin \Phi_1}{\cos \Phi_1}, 
\frac{u\sin \Phi_2}{\cos \Phi_1\cos \Phi_2}, 
\ldots ,\frac{u\sin \Phi_{d-2}}{\cos \Phi_1\ldots \cos \Phi_{d-2}},  
\frac{u\cot \Theta}{\cos \Phi_1\ldots \cos \Phi_{d-2}}-1\right). 
\end{align}
Let $I_{d-1}= [-1,1]^{d-1}$ and let $\mathcal{B}(A)$ denote the family 
of Borel subsets of a set $A$.

\begin{lemma}\label{lemma:cubeSubsetConditionalProbability}
Let $\rho_u$ be the conditional distribution of $\widetilde{Z}_u$ given
$\{\widetilde{Z}_u\in I_{d-1}\}$, i.e.,
\begin{align}\label{j13.2}
\rho_u(A)=\P(\widetilde{Z}_u\in A\mid \widetilde{Z}_u\in I_{d-1}),
\qquad  A\in \mathcal{B}(I_{d-1}).
\end{align}
The measures $\rho_u$
converge weakly, when $u\to \infty$, to the measure $\rho_{\infty}$ given by  
\begin{align}\label{j12.1}
\rho_{\infty}(A) = \frac{1}{2^d}\int_{A}dx_2\ldots dx_{d-1}\, (1+x_d) dx_d.
\end{align}

\end{lemma}

\begin{proof}
Let $I_{t_2,\ldots,t_d}:= [-1,-1+t_2]\times \ldots \times [-1,-1+t_{d}]$. 
It will suffice to prove that, as $u\to\infty$,
\begin{align}
\label{j11.2}
\P(\widetilde{Z}_u\in I_{t_2,\ldots,t_d}|\widetilde{Z}_{u}\in I_{d-1})
=\frac{\P(\widetilde{Z}_u\in I_{t_2,\ldots,t_d})}{\P(\widetilde{Z}_u\in I_{d-1})}
\to \frac{t_2\ldots t_{d-1}t_d^2}{2^d}.
\end{align}

 First, note that if $\widetilde{Z}_u\in I_{d-1}$ then 
$$\left|\frac{u\sin \Phi_k}{\cos \Phi_1\ldots \cos \Phi_k}\right|\leq 1
$$
and, therefore,
$|\sin \Phi_k|\leq 1/u$
for $k=1,\ldots, d-2$. Hence, for any $\varepsilon>0$, for large enough $u>0$,
if $\widetilde{Z}_u\in I_{d-1}$ then  
$\cos \Phi_1\cos \Phi_2\ldots \cos \Phi_{k-2} \in (1-\varepsilon,1]$
for $k =3, \dots , d$.
This implies that for large enough $u$ we have 
\begin{align}\label{eq:inequalityEstimates} 
\P&\left((\sin \Phi_1,\ldots,\sin \Phi_{d-2},\cot\Theta)\in 
\frac{1-\varepsilon}{u}\cdot I_{t_2,\ldots,t_{d-1},1+t_{d}}\right)\\
&\leq \P(\widetilde{Z}_u\in I_{t_2,\ldots,t_d})
 \nonumber \\
&\leq \P\left((\sin \Phi_1,\ldots,\sin \Phi_{d-2},\cot\Theta)\in 
\frac{1}{u}\cdot I_{t_2,\ldots,t_{d-1},1+t_{d}}\right).\nonumber
\end{align}
Since $\Phi_1,\ldots ,\Phi_{d-2}$ and $\Theta$ are independent, 
we have for $\alpha>0$,
\begin{align}\label{j11.3}
 \P&\left((\sin \Phi_1,\ldots,\sin \Phi_{d-2},\cot\Theta)\in \frac{\alpha}{u}\cdot I_{t_2,\ldots,t_{d-1},1+t_{d}}\right)\\
&=\left[\prod_{k=2}^{d-1}\P\left(\sin \Phi_{k-1} \in \frac{\alpha}{u}[-1,-1+t_{k}]\right)\right] \P\left(\cot\Theta \in \frac{\alpha}{u}[0,t_{d}]\right).\nonumber
\end{align}
For $k=2\ldots, d-1$,
\begin{align}\label{j11.4}
 \P&\left(\sin \Phi_{k-1} \in \frac{\alpha}{u}[-1,-1+t_{k}]\right)\\
&=\frac{\arcsin\left( \frac{\alpha(-1+t_{k})}{u}\right)-\arcsin \left(-\frac{\alpha}{u}\right)}{\pi}\sim \frac{\alpha}{u}\cdot \frac{t_k}{\pi}, \qquad \text{  for  } u \to \infty.\nonumber
\end{align}
It follows from (\ref{lem:theta}) that
\begin{align*}
 \P&\left(\cot \Theta \in \frac{\alpha}{u}[0,t_{d}]\right)=
\P\left(\frac{\sqrt{1-V^2}}{V}\leq \frac{\alpha}{u}t_{d}, V\geq 0\right)\\
&=\P\left(1\leq \left(\frac{\alpha^2}{u^2}t_{d}^2+1\right)V^2, 
V\geq 0\right) = \P\left( \sqrt{\frac{1}{1+\frac{\alpha^2}{u^2}t_{d}^2}}
\leq V\right)\\
&=\frac{1}{2}\cdot \left(1-\sqrt{\frac{1}{1+\frac{\alpha^2}{u^2}t_{d}^2}}
\right)\sim \frac{\alpha^2}{4u^2}t_{d}^2,
\qquad \text{  for  } u \to \infty.
\end{align*}
This, \eqref{j11.3} and \eqref{j11.4} yield,  for $u \to \infty$,
\begin{align}\label{j13.1}
\P\left((\sin \Phi_1,\ldots,\sin \Phi_{d-2},\cot\Theta)\in \frac{\alpha}{u}\cdot I_{t_2,\ldots,t_{d-1},1+t_{d}}\right)
\sim \frac{\alpha^{d}}{4u^d\pi^{d-2}}t_2\cdots t_{d-1}t_d^2.
\end{align}
Using the fact that $I_{d-1}=I_{2,\ldots,2}$, we obtain from (\ref{eq:inequalityEstimates}) and the last formula,
$$(1-\varepsilon)^d\frac{t_2\cdots t_{d-1}t_d^2}{2^d}\leq 
\liminf_{u\to\infty}\P(\widetilde{Z}_u\in I_{t_2,\ldots,t_d}|\widetilde{Z}_u
\in I_{d-1})$$
$$\leq \limsup_{u\to\infty}\P(\widetilde{Z}_u\in I_{t_2,\ldots,t_d}|
\widetilde{Z}_u\in I_{d-1})\leq \frac{1}{(1-\varepsilon)^d}\frac{t_2\cdots t_{d-1}t_d^2}{2^d}.$$
As we let $\varepsilon \downarrow 0$ the lemma follows from \eqref{j11.2}.
\end{proof}

\begin{corol}
As $u\to\infty$, we have
\begin{equation}
 \P(\widetilde{Z}_u\in I_{d-1})\sim \frac{2^{d-2}}{u^d\pi^{d-2}}.\label{eq:probabilityQuadrant}
\end{equation}
For $A\in \mathcal{B}(I_{d-1})$ such that $\lambda_{d-1}(\partial A)=0$, when $u\to\infty$,
\begin{equation}
 \P(\widetilde{Z}_u\in A)\sim \frac{1}{4u^d\pi^{d-2}}\int_{A}dx_2\ldots dx_{d-1}(1+x_d)dx_d. \label{eq:probabilityOfSetA}
\end{equation}
\end{corol}

\begin{proof}
The claim $(\ref{eq:probabilityQuadrant})$ follows from  
\eqref{eq:inequalityEstimates} and \eqref{j13.1} 
when we set $t_1=\ldots=t_d=2$.

Claim \eqref{eq:probabilityOfSetA} now follows from \eqref{eq:probabilityQuadrant} and Lemma \ref{lemma:cubeSubsetConditionalProbability}.
\end{proof}

\begin{propo}\label{propo:densityAsymptotics}
For all $u>0$, $\widetilde{Z}_u$ is a continous random vector with density $f_u$ continous at $(0,\ldots,0)$.
Further, we have as $u\to\infty$
$$f_u(0,\ldots, 0)\sim \frac{1}{4u^d\pi^{d-2}}.$$
\end{propo}

\begin{proof}
It is easy to see from the definition that $\widetilde{Z}_u$ 
has a continuous density. By $(\ref{eq:probabilityOfSetA})$ we have 
\begin{align*}
4u^d\pi^{d-2}\P(\widetilde{Z}_u\in A)
&=\int_{A}4u^d\pi^{d-2}f_u(x_2,\ldots,x_d)dx_2\ldots dx_d\\
&\to \int_{A}dx_2\ldots dx_{d-1}(1+x_d)dx_d,
\end{align*}
as $u\to\infty$ for all $A\in \mathcal{ B}(\R^{d-1})$ such that 
$\lambda (\partial A)=0$. This implies $4u^d\pi^{d-2}f_u(0,\ldots,0) \to 1$ as $u\to\infty$.
\end{proof} 

\begin{figure}[ht]
\begin{center}
 \includegraphics[width=8cm]{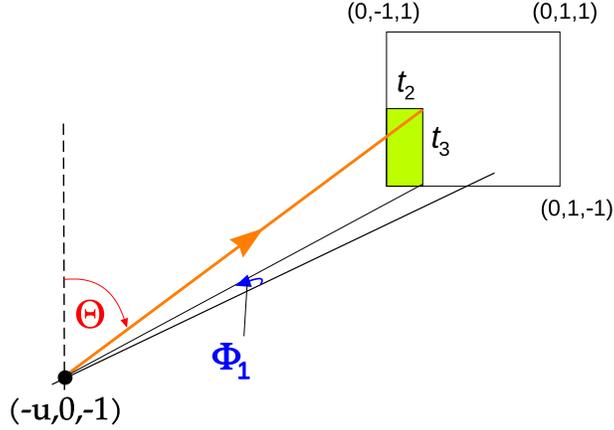}
\caption{Illustration of Lemma \ref{lemma:cubeSubsetConditionalProbability}: Given that the ray passes through the square, the probability that it passes through the rectangle is approximately $t_3^2t_2/8$ for large $u$. }\label{image:rayIntersectsPlane}
\end{center}
\end{figure}

The closed ball in $\R^{d-1}$ with center $x$ and radius $r$ will
be denoted $B_r(x)$. We will write $\D = B_1(0)$ and we will use 
$\lambda_{d-1}$ to denote Lebesgue measure on $\R^{d-1}$.
The following result will help us estimate where the light ray  
exits the tube.

\begin{propo}\label{thm:reflection:plane:intersection}
The measure $\tau_u$ defined by 
$$\tau_{u}(A)=\P(\widetilde{Z}_u\in A\mid \widetilde{Z}_u\in\D)\quad 
\textrm{for} \ A\in \mathcal{B}(\D),
$$ 
converges weakly, as $u\to\infty$, to the  
measure $\tau_{\infty}$ given by  
\begin{equation}
 \tau_{\infty}(A) = 
\frac{\Gamma\left(\frac{d+1}{2}\right)}{\pi^{(d-1)/2}}
\int_{A}dx_2\ldots dx_{d-1}\, (1+x_d) dx_d
\quad \textrm{for} \ A\in \mathcal{B}(\D). \label{eq:tau:infty}
\end{equation}
\end{propo}
\begin{proof}
Recall $\rho_\infty$ defined in \eqref{j12.1}.
We have
\begin{align*}
\rho_{\infty}(\D) 
= \int_{\D}dx_2\ldots dx_{d-1}\, (1+x_d) dx_d
=\int_{\D}dx_2\ldots dx_{d-1} dx_d
=\frac{\pi^{(d-1)/2}}{\Gamma\left(\frac{d+1}{2}\right)}.
\end{align*}
Suppose that $A\in \mathcal{B}(\D)$. Note that $\tau_{\infty}(\partial A) =0$ if and only if 
$\lambda_{d-1}(\partial A)=0 $ if and only if $\rho_{\infty}(\partial A)=0$.
Also note that $\lambda_{d-1}(\partial \D)=0$. These remarks imply that if $\lambda_{d-1}(\partial A)=0$ then 
$$\tau_{u}(A)=\frac{\rho_u(A)}{\rho_u(\D)}\to 
\frac{\rho_{\infty}(A)}{\rho_{\infty}(\D)}=
\tau_{\infty}(A),\quad \textrm{as}\ u\to\infty. $$
Tha claim now follows from Portmanteau's theorem.
\end{proof}

\begin{propo}\label{prop:tauInftyProperties}
Measure $\tau_\infty$ has the following properties.
\begin{enumerate}[(a)]
 \item $\tau_{\infty}(B_r(0))=r^{d-1}$ for $r\leq 1$.
 \item If $y=(0,\ldots,0,-1+r)$ then
$\tau_{\infty}(B_r(y))=r^d$ for $r\leq 1$.
\end{enumerate}
\end{propo}

\begin{proof}
By symmetry,
\begin{align*}
\tau_{\infty}(B_r(0))
&=\frac{\Gamma\left(\frac{d+1}{2}\right)}{\pi^{(d-1)/2}}\int_{B_r(0)}dx_2\ldots dx_{d-1}\,(1+x_d) \, dx_d\\
&= \frac{\Gamma\left(\frac{d+1}{2}\right)}{\pi^{(d-1)/2}}\int_{B_r(0)}dx_2\ldots dx_{d-1} dx_d=
\frac{\Gamma\left(\frac{d+1}{2}\right)}{\pi^{(d-1)/2}}\lambda_{d-1}(B_r(0))
= r^{d-1}. 
\end{align*}
For part (b), we apply the change of coordinates $(h_2,\ldots, h_d-1+r)=(x_1,\ldots, x_d)$, and we obtain
\begin{align*}
 \frac{\Gamma\left(\frac{d+1}{2}\right)}{\pi^{(d-1)/2}}
&\int_{B_r(y)}dx_2\ldots dx_{d-1}\, x_d\, dx_d=
\frac{\Gamma\left(\frac{d+1}{2}\right)}{\pi^{(d-1)/2}}\int_{B_r(0)}dh_2\ldots dh_{d-1}\, (h_d+r)\, dh_d\\
 =&\frac{\Gamma\left(\frac{d+1}{2}\right)}{\pi^{(d-1)/2}} r
\int_{B_r(0)}dh_2\ldots dh_{d-1}\, dh_d+
\frac{\Gamma\left(\frac{d+1}{2}\right)}{\pi^{(d-1)/2}}
\int_{B_r(0)}dh_2\ldots dh_{d-1}\, h_d\, dh_d\\
 =&\frac{\Gamma\left(\frac{d+1}{2}\right)}{\pi^{(d-1)/2}} r \lambda_{d-1}(B_r(0))=r^d .
\end{align*}
\end{proof}

Let $\Sph =\{\sig\in \R^{d-1}: \sig_1^2+\ldots+\sig_{d-1}^2=1\}$. We will analyze the trajectory of a light ray starting from $(-u,\sig)$, where $\sig$ is any point in $\Sph$.
For $\sig\in \Sph $ let 
$U^{\sig}$ be any unitary linear operator on $\R^{d-1}$ such that $U^{\sig}(0,\ldots,-1)=\sig$. The operator $U^{\sig}$ can be identified with an orthogonal matrix. 
We will use ${}^*$ to denote the transpose.
In particular, $U^{\sig *}$ will denote the inverse operator to $U^\sig$.

Let 
\begin{equation}\label{def:ZUSigma}
Z_{u,\sig}=(0,\widetilde{Z}_{u,\sig})
\end{equation}
 be the point where the light ray starting at $(-u,\sig)$ intersects the plane $\{x:x_1=0\}$.

\begin{remark}\label{j13.4}
(i) It is easy to see that for $A\in \mathcal{B}(\D)$, the probability   $\P(\widetilde{Z}_{u}\in U^{\sig*}(A))$ is equal to 
$\P(\widetilde{Z}_{u,\sig}\in A)$. Hence, it depends on $u,\sig$ and $A$ but it does not depend on the choice of $U^\sig$.

(ii)
It follows from the first part of the remark and Proposition \ref{thm:reflection:plane:intersection} that
for $\sig\in \Sph$, probability measures
$\P(\widetilde{Z}_{u,\sig}\in \cdot\mid\widetilde{Z}_{u,\sig}\in \D)$ converge weakly to the probability measure $\tau_{\infty}\circ U^{\sig *}$,
as $u\to \infty$.
\end{remark}

\begin{lemma}
For $\sig\in \Sph$ and $A\in \mathcal{B}(\D)$,
\begin{align}
 &\tau_{\infty}(U^{\sig*}(A))
=\frac{\lambda_{d-1}(A)}{\lambda_{d-1}(\D)}-\frac{1}{\lambda_{d-1}(\D)}\int_{A}(\sig_1x_2+\ldots +\sig_{d-1}x_d)\, dx_2\ldots dx_d, \label{eq:tau_inftyU}\\
 &\int_{\Sph} \tau_{\infty}(U^{\sig*}(A))\P(\Sigma\in d\sig)=\frac{\lambda_{d-1} (A)}{\lambda_{d-1}(\D)},
\label{eq:sphere:integral:tau}
\end{align}
where $\Sigma$ is a random vector uniformly distributed  on $\Sph$.
\end{lemma}

\begin{proof}
 We have from $(\ref{eq:tau:infty})$,
\begin{align*}
 \tau_\infty(U^{\sig*}(A))&= 
\frac{\Gamma\left(\frac{d+1}{2}\right)}{\pi^{(d-1)/2}}
\int_{U^{\sig*}(A)}dx_2\ldots dx_{d-1}(1+x_d)dx_d\\
&= \frac{\Gamma\left(\frac{d+1}{2}\right)}{\pi^{(d-1)/2}}
\int_{U^{\sig*}(A)}dx_2\ldots dx_{d-1}dx_d
+\frac{\Gamma\left(\frac{d+1}{2}\right)}{\pi^{(d-1)/2}}
\int_{U^{\sig*}(A)}x_ddx_2\ldots dx_{d-1}dx_d\\
&= \frac{\lambda_{d-1}(A)}{\lambda_{d-1}(\D)}
+\frac{\Gamma\left(\frac{d+1}{2}\right)}{\pi^{(d-1)/2}}
\int_{A}(U^{\sig*}x)_ddx_2\ldots dx_{d-1}dx_d.
\end{align*}
By definition,  $(U^{\sig*}x)_d= (U^{\sig}(0,\ldots,0,1))^*x= -\sig^*x$. 
This and the last displayed formula imply \eqref{eq:tau_inftyU}.

Fromula $(\ref{eq:sphere:integral:tau})$ follows from the fact that
the integral in $(\ref{eq:tau_inftyU})$
becomes $0$ since we are integrating over a symmetric surface.
\end{proof}

\begin{lemma}\label{lemma:continuity:convergence}
If $A\in \mathcal{B}(\D)$ and $\lambda_{d-1}(\partial A)=0$ then $\P(\widetilde{Z}_{u,\sig}\in A|\widetilde{Z}_{u,\sig}\in \D)$ converges uniformly to $\tau_{\infty}(U^{\sig*}(A))$, as $u\to \infty$.
\end{lemma}

\begin{proof}
Let $A\triangle B$ denote the symmetric difference of sets $A$ and $B$.
It is elementary to prove that for any $A\in \mathcal{B}(\D)$ and $\delta >0$ there exists $\eps>0$ such that if $U:\D \to \D$ is an isometry 
satisfying $|U(x) - x| < \eps$ then  $\lambda_{d-1}(U(A) \triangle A) < \delta$.

It is easy to see that there exists $c_1 < \infty$ such that for every  $u\geq 1$, the density of the measure 
$\P(\widetilde{Z}_{u}\in \, \cdot\mid \widetilde{Z}_{u}\in \D)$
is bounded above by $c_1$ (a formal proof could be based on ideas used in the proof of Lemma \ref{lemma:cubeSubsetConditionalProbability}).

Fix any $A\in \mathcal{B}(\D)$ and $\delta_1 >0$. 
Find $\eps>0$ so small that if $U:\D \to \D$ is an isometry 
satisfying $|U(x) - x| < \eps$ then  $\lambda_{d-1}(U(A) \triangle A) < \delta_1/c_1$.
Suppose that $y, z \in \Sph$ and $|y-z| < \eps$. Then there exists an isometry $U:\D \to \D$ such that $U(y) = z$ and $|U^*(x) - x| < \eps$ for all $x\in \D$. We now  use Remark \ref{j13.4} (i) to see that 
\begin{align*}
\P(\widetilde{Z}_{u,\sig}\in A\mid \widetilde{Z}_{u,\sig}\in \D)
&=
\P(\widetilde{Z}_{u}\in U^{\sig *}(A)\mid \widetilde{Z}_{u}\in \D)\\
&=
\P(\widetilde{Z}_{u}\in (U\circ U^{\sig })^*(U^*(A))\mid \widetilde{Z}_{u}\in \D)\\
&\leq
\P(\widetilde{Z}_{u}\in (U\circ U^{\sig })^*(A)\mid \widetilde{Z}_{u}\in \D) + (\delta_1/c_1) c_1\\
&=
\P(\widetilde{Z}_{u,z}\in  U^{z * }(A)\mid \widetilde{Z}_{u,z}\in \D) + \delta_1.
\end{align*}
We see that the functions $y \to \P(\widetilde{Z}_{u,\sig}\in A\mid \widetilde{Z}_{u,\sig}\in \D)$ are equicontinuous for $u\geq 1$. 
The lemma now follows from the pointwise convergence of
$\P(\widetilde{Z}_{u,\sig}\in A\mid \widetilde{Z}_{u,\sig}\in \D)$  to $\tau_{\infty}(U^{\sig*}(A))$, as $u\to \infty$, proved 
in Proposition \ref{thm:reflection:plane:intersection}
and Remark \ref{j13.4} (ii).
\end{proof}

\begin{propo}\label{prop:unifrom:ratio}
If $A\in \mathcal{B}(\D)$, $\lambda_{d-1}(A)>0$ and $\lambda_{d-1}(\partial A)=0$ then the ratio 
$$\frac{\P(\widetilde{Z}_{u,\sig}\in A\mid \widetilde{Z}_{u,\sig}\in \D)}{\tau_{\infty}(U^{\sig*}(A))}$$
converges uniformly to $1$ as $u\to\infty$.
\end{propo}

\begin{proof}
It is easy to see that $\tau_{\infty}(U^{\sig*}(A))>0$ for every $\sig\in\Sph$ because  $\lambda_{d-1}(A)>0$.
It follows from \eqref{eq:tau_inftyU} that the function $\sig\mapsto \tau_{\infty}(U^{\sig*}(A))$  is continuous.
Therefore, due to compactness of $\Sph$, 
we have $m:=\inf\{\tau_{\infty}(U^{\sig*}(A)):\sig\in \Sph\}>0$. Hence, 
\begin{align*}
 &\left|\frac{\P(\widetilde{Z}_{u,\sig}\in A\mid \widetilde{Z}_{u,\sig}\in \D)}{\tau_{\infty}(U^{\sig*}(A))}-1\right|
\leq \frac{1}{m}|\P(\widetilde{Z}_{u,\sig}\in A\mid \widetilde{Z}_{u,\sig}\in \D)-\tau_{\infty}(U^{\sig*}(A))|,
\end{align*}
and by Lemma \ref{lemma:continuity:convergence}  we have uniform convergence.
\end{proof}

\section{Light ray reflections}\label{sec:construction}

In our model, 
the reflection points of a light ray form a Markov chain with the state space being the tube
$$\mathcal{S}=\{x\in \R^d\ :\ x_1\in \R, x_2^2+\ldots +x_d^2 =1\}.$$ 
To  simplify notation, we  shift the coordinate system of our model so that 
the light source is located in the hyperplane $\{x:x_1=0\}$, and the end of the tube is located at $\{x:x_1=s\}$, with $s>0$. 

We will now provide a formal description of our model. 
Let $\Sigma$ be a random vector uniformly distributed on the $(d-1)$-dimensional sphere $\Sph$.
Let $(\Theta^{(k)})_{k\geq 1}$, $(\Phi_1^{(k)})_{k\geq 1}$, \ldots, $(\Phi_{d-2}^{(k)})_{k\geq 1}$
be sequences such that:
\begin{itemize}
\item Each of these sequences is i.i.d.;
 \item $\Theta^{(k)}$ has the density $\frac{1}{2}\cos \theta$ on $[-\pi/2,\pi/2]$ for  $k\geq 1$;
 \item $\Phi_j^{(k)}$ is distributed as $U[-\pi/2,\pi/2]$ for all $k\geq 1$ and $j=1,\ldots, d-2$;
 \item All random variables $\Sigma$, $(\Theta^{(k)})_{k\geq 1}$, $(\Phi_1^{(k)})_{k\geq 1}$, \ldots, $(\Phi_{d-2}^{(k)})_{k\geq 1}$
are jointly independent.
\end{itemize}

We define a Markov chain $\{S_k, k\geq 0\}$  by setting $S_0=(0,\Sigma)$,
 and for $k\geq 1$,
\begin{align}
R_k&=\frac{2\cos\Theta^{(k)}}{1-\left(\sin\Theta^{(k)}\cos \Phi_{d-2}^{(k)}\ldots\cos\Phi_1^{(k)}\right)^2},
\label{j14.1}\\
S_k &= (S^{x_1}_k, S^{x_{2,\ldots, d}}_k), \text{  where  }
S^{x_1}_k \in \R, S^{x_{2\ldots d}}_k \in \Sph, \nonumber \\
 S^{x_1}_k&=S^{x_1}_{k-1}+R_{k}\sin \Theta^{(k)}\cos \Phi_1^{(k)}\cdots \cos \Phi_{d-2}^{(k)},
\label{j14.2}\\
S^{x_{2,\ldots, d}}_k&=U^{S^{x_{2,\ldots, d}}_{k-1}}(G_{2,\ldots, d}(\Theta^{(k)}, \Phi_1^{(k)},\ldots ,\Phi_{d-2}^{(k)})),\nonumber
\end{align}
where:
\begin{itemize}
 \item $U^{x_{2,\ldots, d}}$ is a unitary operator that maps $(0, \ldots, 0,-1)\in \R^{d-1}$ 
to $x_{2,\ldots, d}\in \Sph$;
 \item $G_{2,\ldots, d}(\Theta, \Phi_1,\ldots ,\Phi_{d-2})$ represents the last $(d-1)$
coordinates of $(\ref{eq:reflvector})$ where $R$ is given by $(\ref{eq:radius})$.
\end{itemize}

\begin{lemma}\label{prop:markovChainRandomWalk}
\begin{enumerate}[(a)]
 \item $(S^{x_1}_k)$ is a symmetric random walk.
 \item The uniform distribution on $\Sph$ is an invariant distribution 
for the Markov chain $S^{x_{2,\ldots, d}}_k$.
 \item For every random time $T$ measurable with respect to the $\sigma$-field $\sigma(S_k^{x_1}:k\geq 0)$, 
$S^{x_{2,\ldots, d}}_T$ is uniformly distributed on  $\Sph$ and independent of $\sigma(S_k^{x_1}:k\geq 0)$.
\end{enumerate}
\end{lemma}
\begin{proof} Part (a) follows from the definition of the model. Parts 
(b) and (c) follow from the fact that the initial position $S_0$ is governed by a random variable $\Sigma$, independent of all other random variables used in the construction.
\end{proof}

From now on, we will use $N_s$, $O_s$ and $U_s$ to denote random variables defined in \eqref{eq:NsMk}, but relative to the random walk $S^{x_1}_k$.
For future reference we record a formula for the position of the exit point
 $(s,Y_s)$ of the light ray from the tube.

\begin{lemma}\label{j15.1}
The light ray exits the tube at a point $(s,Y_s)$ where 
\begin{equation}
 Y_s=S^{x_{2,\ldots, d}}_{N_s-1}+\frac{s-S^{x_1}_{N_s-1}}{S^{x_1}_{N_s}-S^{x_1}_{N_s-1}}\left(S^{x_{2,\ldots, d}}_{N_s}- S^{x_{2,\ldots, d}}_{N_s-1}\right).
\end{equation}
\end{lemma}

\begin{proof}
Note that $(s,Y_s)$ lies on the line segment $[S_{N_s-1},S_{N_s}]$.  
\end{proof}

\section{The tail of the step in $x_1$ direction}\label{sec:tail}

We will study the distribution of the step $X_k=S^{x_1}_k-S^{x_1}_{k-1}$
of the random walk $S^{x_1}_k$. To simplify notation, we will 
suppress $k$ in the notation, for example, we will write $X$ instead of $X_k$. In view of \eqref{j14.1} and \eqref{j14.2}, we may represent $X$ as
\begin{equation}\label{eq:X_Nd:def}
 X=\frac{2\cos\Theta\sin\Theta\cos \Phi_{d-2}\ldots\cos \Phi_1}{1-\sin\Theta^2\cos^2 \Phi_{d-2}\ldots\cos^2 \Phi_1},
\end{equation}
where, as usual, $\Theta$ has density \eqref{eq1}, $\Phi_1,\ldots, \Phi_{d-2}$ are $U(-\pi/2,\pi/2)$
and they are all independent.

\begin{propo}\label{prop:step3d} 
We have for $x>0$,
\begin{equation}
 \label{eq:tailX:Nd}
 \P(X>x)= \int_{\frac{x}{\sqrt{4+x^2}}}^1
 \P\left(\cos\Phi_{d-2}\ldots\cos \Phi_1>\frac{\sqrt{1-v^2+x^2}-\sqrt{1-v^2}}{xv}\right)dv.
\end{equation}
In the 3-dimensional case we have for $x>0$,
\begin{equation}\label{eq:tailX:3d}
 \P(X>x)= \frac{2}{\pi}\int_{\frac{x}{\sqrt{4+x^2}}}^1
 \cos^{-1}\left(\frac{\sqrt{1-v^2+x^2}-\sqrt{1-v^2}}{xv}\right)dv.
\end{equation}
\end{propo}

\begin{proof}
 Let $V=\sin\Theta$, $F_k=\cos \Phi_k$ for $k=1,\ldots d-2$ and $\mathbf{F}=F_1\cdots F_{d-2}$. 
We obtain from the representation given 
in $(\ref{eq:X_Nd:def})$,
\begin{align*}
  \P(X>x)&= \P(X>x,V>0)\\ 
&=\P\left(xV^2\bF^2+2V\sqrt{1-V^2}\bF-x>0,V>0\right)\\
&=  \P\left(\bF>\frac{\sqrt{1-V^2+x^2}-\sqrt{1-V^2}}{xV},V>0\right)\\
&= \P\left(\bF>\frac{\sqrt{1-V^2+x^2}-\sqrt{1-V^2}}{xV}\wedge 1,V>0\right).
\end{align*}
The inequality $\frac{\sqrt{1-v^2+x^2}-\sqrt{1-v^2}}{xv}\leq 1$ holds for $v\in (0,1)$ and fixed $x>0$
if and only if $v\geq \frac{x}{\sqrt{4+x^2}}$. 
Recall that $V$ has the uniform distribution $U(-1,1)$.
Hence, we get
$$\P(X>x)= \int_{\frac{x}{\sqrt{4+x^2}}}^1\P\left(\bF>\frac{\sqrt{1-v^2+x^2}-\sqrt{1-v^2}}{xv}\right)dv,$$
completing the proof of \eqref{eq:tailX:Nd}.
In the 3-dimensional case we have ${\bf F}=F_1=\cos\Phi_1$
\begin{align*}
\P(X>x)&=
 \int_{\frac{x}{\sqrt{4+x^2}}}^1 {\textstyle
\P\left(\Phi_1 \in \left[-\cos^{-1}\left(\frac{\sqrt{1-v^2+x^2}-\sqrt{1-v^2}}{xv}\right),\cos^{-1}\left(\frac{\sqrt{1-v^2+x^2}-\sqrt{1-v^2}}{xv}\right)\right]\right)}dv\\
&=  \int_{\frac{x}{\sqrt{4+x^2}}}^1{\textstyle \frac{2}{\pi}\cos^{-1}\left(\frac{\sqrt{1-v^2+x^2}-\sqrt{1-v^2}}{xv}\right)}dv.
\end{align*}
\end{proof}

The following is the main result of this section.

\begin{thm}\label{the:theorem}
If the dimension of the space is $d\geq 3$ then we have 
\begin{align*}
\lim_{x\to\infty}x^d\P(X>x)=
\begin{cases}
\displaystyle
\frac{2}{(d-1)!!}\left(\frac{2}{\pi}\right)^{ (d-2)/2}
&
\text{  if  $d$ is even},\\
\displaystyle
\frac{2}{(d-1)!!}\left(\frac{2}{\pi}\right)^{(d-3)/2}
&
\text{  if  $d$ is odd}.
\end{cases}
\end{align*}
\end{thm}

\begin{proof}
The following formula is a reformulation of 
Proposition \ref{thm:prodarcsine},
\begin{equation}\label{j14.3}
\lim_{x\to 1^-}(1-x)^{-d/2}\P\left(\cos\Phi_{d}\ldots\cos \Phi_1>x\right)=c_d:=
\begin{cases}
\displaystyle
\frac{1}{d!!}\left(\frac{4}{\pi}\right)^{d/2} & \text{  if $d$ is even}, \\
\displaystyle
\frac{1}{\sqrt{2}\, d!!}\left(\frac{4}{\pi}\right)^{ (d+1)/2} & \text{  if $d$ is odd} .
\end{cases}
\end{equation}

We will show that
\begin{align}\label{j11.1}
\lim_{x\to\infty}\sup_{v\in\left[\frac{x}{\sqrt{x^2+4}},1\right]}\left|1-\frac{\sqrt{1-v^2+x^2}-\sqrt{1-v^2}}{xv}\right|=0.
\end{align}
If
\begin{align*}
A(v,x)&=\frac{2x\sqrt{1+v}}{x+\sqrt{1-v^2}+\sqrt{1-v^2+x^2}},\\
B(v,x)& = \frac{x}{v(\sqrt{1-v^2+x^2}+\sqrt{1-v^2})},\nonumber
\end{align*}
then
\begin{align}\label{eq:dif:A:B}
 1-\frac{\sqrt{1-v^2+x^2}-\sqrt{1-v^2}}{xv}
=\frac{\sqrt{1-v}}{x} A(v,x)
+ (v-1) B(v,x).
\end{align}
It is not hard to show that 
\begin{equation*}
 \lim_{x\to\infty}\sup_{v\in\left[\frac{x}{\sqrt{x^2+4}},1\right]}|A(v,x)-\sqrt{2}|=0\quad \textrm{and}\quad \lim_{x\to\infty}\sup_{v\in\left[\frac{x}{\sqrt{x^2+4}},1\right]}|B(v,x)-1|=0.\label{eq:AB:limits}
\end{equation*}
Formula \eqref{j11.1} follows from this and \eqref{eq:dif:A:B}.

It follows from \eqref{eq:tailX:Nd}, \eqref{j14.3} and \eqref{j11.1} that, as $x\to \infty$,
\begin{align}\label{j11.2}
\P(X>x)\sim c_{d-2}\int_{\frac{x}{\sqrt{4+x^2}}}^1{\textstyle \left(1-\frac{\sqrt{1-v^2+x^2}-\sqrt{1-v^2}}{xv}\right)^{\frac{d-2}{2}}}dv.
\end{align}

Substitutions $\frac{1}{v^2}=1+\frac{w}{x^2}$ and $\varepsilon = 1/x^2$ yield
\begin{align}\label{j11.3}
&\int_{\frac{x}{\sqrt{4+x^2}}}^1{\textstyle \left(1-\frac{\sqrt{1-v^2+x^2}-\sqrt{1-v^2}}{xv}\right)^{k}}dv\\
&=
\frac{\varepsilon}{2} \int_0^4\left(1-\sqrt{1+\varepsilon(1+\varepsilon)w}+\varepsilon\sqrt{ w}\right)^{k}(1+\varepsilon w)^{-3/2}dw.\nonumber
\end{align}

Let  $f_w(\varepsilon):=-\sqrt{1+\varepsilon(1+\varepsilon)w}+\varepsilon\sqrt{ w}$ 
and note that $f_w(0)=1$ and $f_w'(0)=\sqrt{w}-\frac{w}{2}$.
We use the Dominated Convergence Theorem to obtain 
$$
\lim_{\varepsilon\to 0} \int_0^4\left(\frac{f_w(\varepsilon)
-f_w(0)}{\varepsilon}\right)^{k}(1+\varepsilon w)^{-3/2}dw= \int_0^4f'_w(0)dw.
$$
This implies
\begin{align}\label{j11.4}
\lim_{\varepsilon\to 0} \frac 12 \varepsilon^{-k} \int_0^4\left(1-\sqrt{1+\varepsilon(1+\varepsilon)w}+\varepsilon\sqrt{ w}\right)^{k}(1+\varepsilon w)^{-3/2}dw= 
\frac 12 \int_0^4\left(\sqrt{w}-\frac{w}{2}\right)^kdw.
\end{align}
The substitution $y=\sqrt{w}/2$ leads to
\begin{align*}
\frac 12 \int_0^4\left(\sqrt{w}-\frac{w}{2}\right)^k\, dw &= \int_0^1 (2y-2y^2)^k 4 ydy
=2^{k+2}\int_0^1y^{k+1}(1-y)^{k}dy .
\end{align*}
This and the identity $\frac{\Gamma(q)\Gamma(p)}{\Gamma(p+q)}=\int^1_0z^{q-1}(1-z)^{p-1}dz$ imply that
for $k> -1 $ we have 
$$\frac 12 \int_0^4\left(\sqrt{w}-\frac{w}{2}\right)^kdw = 2^{k+2}\frac{\Gamma(k+2)\Gamma(k+1)}{\Gamma(2k+3)}.$$
We use this formula, \eqref{j11.3} and \eqref{j11.4} to see that
$$\lim_{x\to\infty}x^d\int_{\frac{x}{\sqrt{4+x^2}}}^1{\textstyle \left(1-\frac{\sqrt{1-v^2+x^2}-\sqrt{1-v^2}}{xv}\right)^{\frac{d-2}{2}}}dv=\frac{2^{d/2}\Gamma(d/2)^2d}{\Gamma(d+1)}.$$
This and \eqref{j11.2} imply that $\P(X>x)\sim \frac{2^{d/2}\Gamma(d/2)^2d}{\Gamma(d+1)}c_{d-2} x^{-d}$, as $x\to \infty$.
The theorem follows from an application of the
formulas $\Gamma(d/2)=2^{-\frac{d-1}{2}}(d-2)!!\sqrt{\pi}$ for odd $d$  and $\Gamma(d/2)=2^{-\frac{d-2}{2}}(d-2)!!$ for even $d$.
\end{proof}

 Theorem \ref{the:theorem} says that $x\mapsto P(X>x)$ is a regularly
varying function with the index $-d$. This implies the following results.

\begin{corol}\label{cor:finiteMoments}
For all $k=0,\ldots , d-1$,  the $k$-th moment of  $X$ exists and is finite. However, the $d$-th moment of  $|X|$ is infinite.
\end{corol}

\begin{corol}
If $Z_1=S^{x_1}_{N_0}$ then $x\mapsto \P(Z_1>x)$ is a regularly varying function of degree $-(d-1)$.
Hence, for all $k=0,1,\ldots, d-2$, the $k$-th moment of $Z_1$ is finite but the $(d-1)$-st moment is not.
\end{corol}

\begin{proof}
See \cite{doney}, which builds on results of  \cite{veraveb}. 
\end{proof}

\begin{corol}
The random walk $(S_k^{x_1})$ is neighborhood recurrent. The light ray will eventually exit the tube almost surely.
\end{corol}
\begin{proof}
By Corollary \ref{cor:finiteMoments}, $X$ has a finite expectation. Since $X$ is symmetric, $\E X=0$. Neighborhood recurrence of $(S_k^{x_1})$ can now be proved as in \cite[Lemma 4.2.]{laser}, using the Chung-Fuchs Theorem. Neighborhood recurrence of $(S_k^{x_1})$ implies that the process will
eventually take a value larger than $s$, a.s. In other words,  the light ray will exit the tube.
\end{proof}

Recall the definition of $U_s$ and $O_s$ from $(\ref{eq:NsMk})$. In dimension $d\geq 3$ we have $\E[X_1^2]<\infty$  and $\E[Z_1]<\infty$.
Proposition 5.5 and Lemma 5.6 in \cite{laser} were concerned with the case $d=3$ but their proofs used only the finiteness of $\E[Z_1]$ and not any other consequences of the assumption that $d=3$. Hence, the results and their proofs apply in the case of our current model, i.e., they hold true for all $d\geq 3$. We state the two results without proofs.

\begin{propo}\label{prop:limit3dU/O}
 For $t\in [0,1]$,
\begin{align}\label{u19.5}
\lim_{s\to\infty}\P\left(\frac{U_s}{U_s+O_s}\leq t\right)
= \Lambda(t) :=
\frac{\E[(t(U_0+O_0)-U_0)^+]}{\E[O_0]}.
\end{align}
\end{propo}

\begin{lemma}\label{lema:r:2}
The function $\Lambda:[0,1]\to \R$ defined in \eqref{u19.5}
has the following properties.
\begin{enumerate}[(a)]
 \item $\Lambda$ is a continuous, increasing and convex function.
 \item $ \frac{\E(X\1_{(X>0)})}{\E[O_0]}\ t\leq \Lambda(t)\leq t$.
 \item $\Lambda(0)=0$, $\Lambda(1)=1$, $\Lambda'(0)=\E(X\1_{(X>0)}) / \E[O_0]$ and $\Lambda'(1)=\E[O_0+U_0]/\E O_0$.
\end{enumerate}
\end{lemma}

For $y_0 \in \R^{d-1}$, let $D_t(\sig_0)= \{\sig\in \R^{d-1}: \|\sig - (1-t)\sig_0\|\leq t \}$.
Note that  $D_t(\sig_0)$ is an $(d-1)$-dimensional ball with measure 
$\frac{\pi^{(d-1)/2}}{\Gamma(\frac{d+1}{2})} t^{d-1}$.
The definition of $D_t(\sig_0)$, \eqref{u19.5} and 
\begin{equation}
 \left\{\frac{U_s}{U_s+O_s}\leq t\right\}
=\left\{ Y_s \in D_t(S^{x_{2,\ldots, d}}_{N_s-1})\right\} \label{eq:underOverShootExit}
\end{equation}
\begin{figure}
\begin{center}
 \includegraphics[width=8cm]{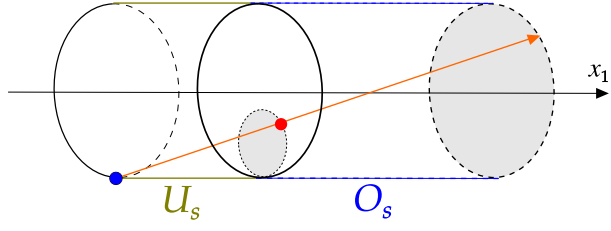}
\caption{Illustration of the equality $(\ref{eq:underOverShootExit})$.}\label{3dOvrUnd}
\end{center}
\end{figure}
yield
\begin{align}\label{u19.4}
\lim_{s\to \infty} \P ( Y_s \in D_t(S^{x_{2,\ldots, d}}_{N_s-1}) )
= \Lambda(t).
\end{align}
 See Figure \ref{3dOvrUnd} for the illustration of the equality (\ref{eq:underOverShootExit}).

\begin{propo} \label{thm:limit:3d:from:direction}
Let $\lambda(A)$ denote the Lebesgue measure on $\partial _r C$.
\begin{enumerate}[(a)]
\item
\begin{align*}
\lim_{t\to 0^+}\frac{\lim_{s\to \infty} \P ( Y_s \in D_t(S^{x_{2,\ldots, d}}_{N_s-1}))}{\lambda(D_t(S^{x_{2,\ldots, d}}_{N_s-1}))}=\infty.
\end{align*}
\item
\begin{align*}
\lim_{t\to 1^-}\frac{\lim_{s\to \infty} \P ( Y_s \in \partial _r C \setminus  D_t(S^{x_{2,\ldots, d}}_{N_s-1}))}{\lambda(\partial _r C \setminus D_t(S^{x_{2,\ldots, d}}_{N_s-1}))}=\frac{\Gamma(\frac{d+1}{2})}{(d-1)\pi^{\frac{d-1}{2}}}\E[O_0+U_0]/\E[O_0].
\end{align*}
\end{enumerate}
\end{propo}

\begin{proof}
(a)
We have 
$\lambda(D_t(S^{x_{2\ldots d}}_{N_s-1}))=\frac{\Gamma(\frac{d+1}{2})}{\pi^{\frac{d-1}{2}}} t^{d-1}$ so part (a) follows from Lemma \ref{lema:r:2} (b) and \eqref{u19.4}.

\medskip
(b) By \eqref{u19.4},
\begin{align*}
\lim_{t\to 1^-}&\frac{\lim_{s\to \infty} \P ( Y_s \in
\partial _r C \setminus D_t(S^{x_{2,\ldots, d}}_{N_s-1}))}{\lambda(\partial _r C \setminus D_t(S^{x_{2,\ldots, d}}_{N_s-1}))}
=\lim_{t\to 1^-}\frac{1-\Lambda(t)}{\lambda_{d-1}(\D)(1-t^{d-1})}\\
&=\lim_{t\to 1^-}\frac{\Gamma(\frac{d+1}{2})}{\pi^{\frac{d-1}{2}}}\cdot\frac{1}{(1+t+\ldots+t^{d-2})}\frac{\Lambda(1)-\Lambda(t)}{1-t}=\frac{\Gamma(\frac{d+1}{2})}{(d-1)\pi^{\frac{d-1}{2}}}\Lambda'(1).
\end{align*}
Part (b) now follows from Lemma \ref{lema:r:2} (c).
\end{proof}

\begin{propo}\label{j12.7}
For $r\in (0,1)$,
\begin{equation}
\label{j12.8}
\lim_{s\to \infty} \P ( Y_s \in B_r(0))\leq \Lambda\left(\frac{1+r}{2}\right)-\Lambda\left(\frac{1-r}{2}\right), 
\end{equation}
where  $\Lambda$ is given by $(\ref{u19.5})$.
\end{propo}
\begin{proof}
Since 
$B_r(0) \subset D_{(1+r)/2}(\sig)\setminus D_{(1-r)/2}(\sig)$ for any $\sig\in \Sph$, we obtain
\eqref{j12.8} by applying \eqref{u19.4}.
\end{proof}

\begin{propo}\label{thm:lim:conditioned}
For $t\in (0,1)$ and $\varepsilon>0$
\begin{equation}
\label{eq:12.8}
\lim_{s\to \infty} \P ( Y_s \in D_t(S_{N_s-1}^{x_{2,\ldots, d}})
\mid S_{N_s-1}^{x_1}\leq s(1-\varepsilon))= t^{d}.
\end{equation}
\end{propo}
\begin{proof}
 The formula follows from $(\ref{eq:underOverShootExit})$ and Theorem \ref{thm:overUnderShoot}.
\end{proof}

\section{Light ray exit distribution}\label{sec:exit}
In this section we discuss the asymptotic properties of the distribution of $Y_s$, the light ray exit point, for large $s$. 

The following formula follows from Lemma \ref{j15.1}
and the construction given in Section  \ref{sec:construction}, 
\begin{align*}
Y_s&=  U^{S^{x_{2,\ldots, d}}_{N_s-1}}\left((0,\ldots,0,-1)\right)\\
&\qquad{\textstyle +\,U^{S^{x_{2,\ldots, d}}_{N_s-1}}\left(\frac{s-S^{x_1}_{N_s-1}}{S^{x_1}_{N_s}-S^{x_1}_{N_s-1}}\left(G_{2,\ldots ,d}(\Theta^{(N_s)}, \Phi_1^{(N_s)},\ldots ,\Phi_{d-2}^{(N_s)})-(0,\ldots,0,-1) \right)\right)}.
\end{align*}
Denote 
\begin{equation}
 Y_s^0=(0,\ldots,0,-1)+\frac{s-S^{x_1}_{N_s-1}}{S^{x_1}_{N_s}-S^{x_1}_{N_s-1}}\left(G_{2,\ldots ,d}(\Theta^{(N_s)}, \Phi_1^{(N_s)},\ldots ,\Phi_{d-2}^{(N_s)})-(0,\ldots,0,-1)\right).
\end{equation}
Note, that $Y_s=U^{S^{x_{2\ldots d}}_{N_s-1}}(Y_s^0)$ and recall the definition of $\tau_{\infty}$ from $(\ref{eq:tau:infty})$.

\begin{thm}\label{lem:tauInftyFromAway}
The distribution of $\P(Y_s^0\in \cdot \mid S_{N_s-1}^{x_1} \leq s - \beta(s) )$
converges weakly to $\tau_\infty$ for every deterministic function 
$\beta$ such that $\lim_{s\to\infty}\beta(s)=\infty$.
\end{thm}
\begin{proof}
It is clear that $\lambda_{d-1}(A)=0$ if and only if $\tau_{\infty}(A)=0$, for all 
$A\in \mathcal{B}(\D)$. Assume that $A\in \mathcal{B}(\D)$ is such that 
$\lambda_{d-1}(\partial A)=0$. We have 
\begin{align}\label{j19.1}
 \P&(Y_s^0\in A,S_{N_s-1} \leq s - \beta(s))=\sum_{k=1}^{\infty}\P(Y_s^0\in A, S_{k-1}^{x_1} \leq s - \beta(s),N_s=k)\\
&=\sum_{k=1}^{\infty}\P(Y_s^0\in A, S_k^{x_1}>s,S_{k-1}^{x_1} \leq s - \beta(s),S_{k-2}^{x_1}\leq s,\ldots, S_1^{x_1}\leq s).\nonumber
\end{align}
Recall from $(\ref{eq:planeIntersectionCoordinates})$ that the ray started at $(-u,0,\ldots,0,-1)$ 
whose direction is governed by angles $\Theta^k$, $\Phi_1^k,\ldots, \Phi_{d-2}^k$ intersects the plane 
at 
$$\widetilde{Z}^{k}_u=\left(\frac{u\sin \Phi_1^{(k)}}{\cos \Phi_1^{(k)}},  \ldots ,\frac{u\sin \Phi_{d-2}^{(k)}}{\cos \Phi_1^{(k)}\ldots \cos \Phi_{d-2}^{(k)}},  \frac{u\cot \Theta^{(k)}}{\cos \Phi_1^{(k)}\ldots \cos \Phi_{d-2}^{k}}-1\right). $$
It follows from \eqref{j19.1} that
\begin{align*}
\P&(Y_s^0\in A,S_{N_s-1} \leq s - \beta(s))\\
 &=\sum_{k=1}^{\infty}\P(\widetilde{Z}^k_{s-S_{k-1}^{x_1}}\in A,S_{k-1}^{x_1} \leq s - \beta(s),S_{k-2}^{x_1}\leq s,\ldots, S_1^{x_1}\leq s)\\
&=\sum_{k=1}^{\infty}\int_{\beta(s)}^{\infty}\P(\widetilde{Z}^k_{u}\in A)\P(s-S_{k-1}^{x_1}\in du ,S_{k-2}^{x_1}\leq s,\ldots, S_1^{x_1}\leq s)\\
&=\sum_{k=1}^{\infty}\int_{\beta(s)}^{\infty}\P(\widetilde{Z}_{u}^k\in A|\widetilde{Z}_{u}^k\in \D)\P(\widetilde{Z}_{u}^k\in \D)\P(s-S_{k-1}^{x_1}\in du ,S_{k-2}^{x_1}\leq s,\ldots, S_1^{x_1}\leq s)\\
&=\sum_{k=1}^{\infty}\int_{\beta(s)}^{\infty}\P(\widetilde{Z}_{u}\in A|\widetilde{Z}_{u}\in \D)\P(\widetilde{Z}_{u}^k\in \D)\P(s-S_{k-1}^{x_1}\in du ,S_{k-2}^{x_1}\leq s,\ldots, S_1^{x_1}\leq s).
\end{align*}
By Proposition \ref{thm:reflection:plane:intersection} we have
for $s\to \infty$,
\begin{align*}
\P&(Y_s^0\in A,S_{N_s-1} \leq s - \beta(s))\\
&\sim \sum_{k=1}^{\infty}\int_{\beta(s)}^{\infty}\tau_\infty(A)\P(\widetilde{Z}_{u}^k\in \D)\P(s-S_{k-1}^{x_1}\in du ,S_{k-2}^{x_1}\leq s,\ldots, S_1^{x_1}\leq s)\\
&= \tau_\infty(A)\P(S_{N_s-1}\leq s-\beta(s)).
\end{align*}
\end{proof}

\begin{corol}\label{cor:rotationExitingPosition}
For $A\in \mathcal{B}(\D)$,
\begin{equation}
\label{eq:12.8c}
\lim_{s\to \infty} \P ( Y_s \in U^{S_{N_s-1}^{x_{2\ldots d}}}(A) \mid S_{N_s-1}^{x_1}\leq s - \beta(s))= \tau_{\infty}(A).
\end{equation}
\end{corol}
\begin{proof}
The result holds since $Y_s=U^{S_{N_s-1}^{x_{2\ldots d}}}(Y_s^0)$ and  $U^{S_{N_s-1}^{x_{2\ldots d}}}$ is an invertible linear operator.
\end{proof}

\begin{remark} Proposition \ref{thm:lim:conditioned} can be derived from Corollary \ref{cor:rotationExitingPosition}.
Since $D_t(S_{N_s-1}^{x_{2,\ldots ,d}})=U^{S_{N_s-1}^{x_{2,\ldots, d}}}D_t((0,0,\ldots,0,-1))$, formula (\ref{eq:12.8})  follows from Proposition \ref{prop:tauInftyProperties} (b).
\end{remark}

\begin{thm}\label{thm:uniformFromAway}
The distribution of $\P(Y_s\in \cdot \mid S_{N_s-1}^{x_1} \leq s - \beta(s) )$
converges weakly to the uniform distribution on $\D$ for every deterministic function 
$\beta$ such that $\lim_{s\to\infty}\beta(s)=\infty$.
\end{thm}

The theorem is illustrated in Figures \ref{j26.2} and \ref{j26.1}. 

\begin{figure}[ht] \label{j26.1}
\begin{center}
 \includegraphics[width=6cm]{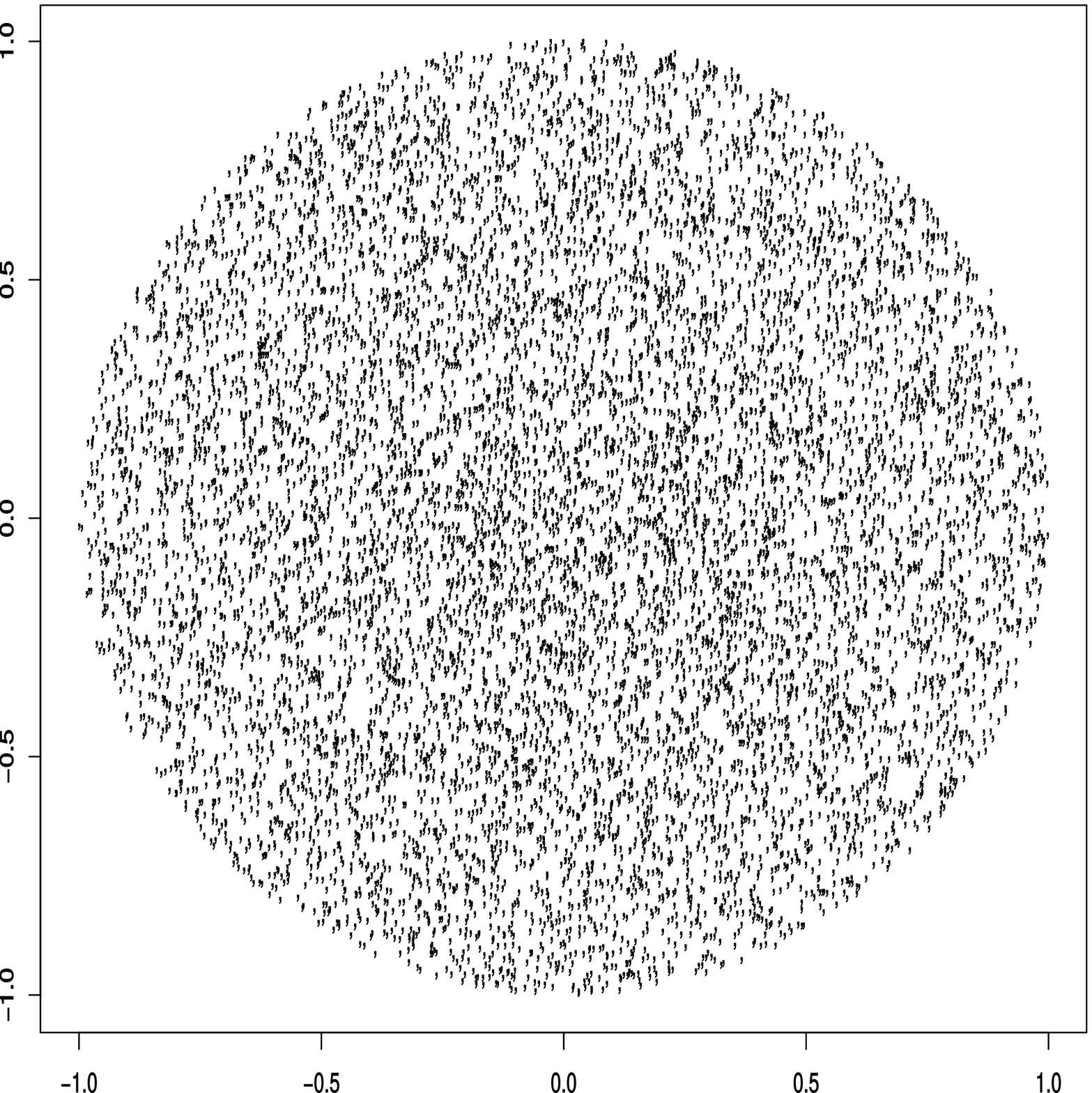}\quad \includegraphics[width=6cm]{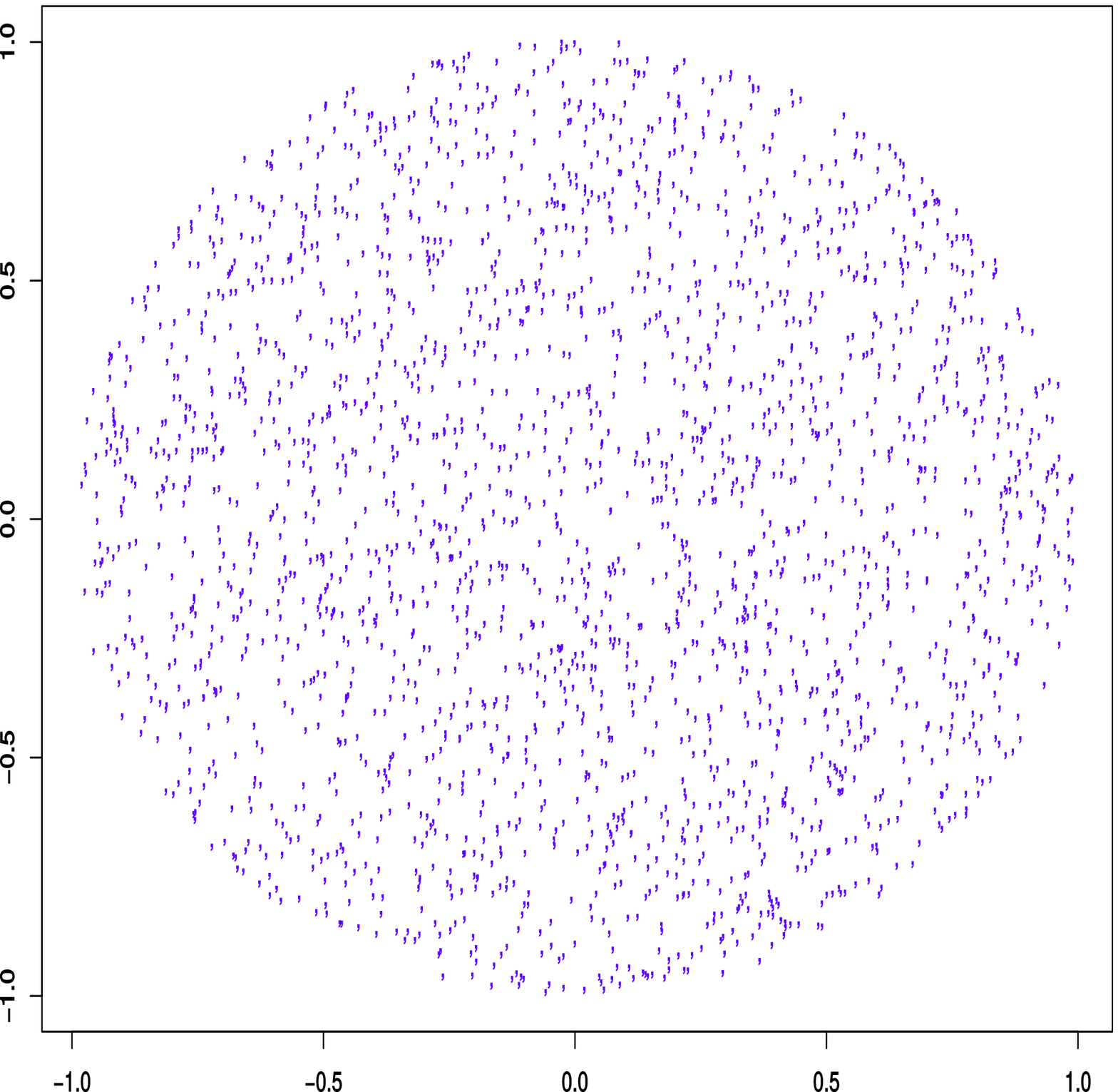}\\
 \includegraphics[width=6cm]{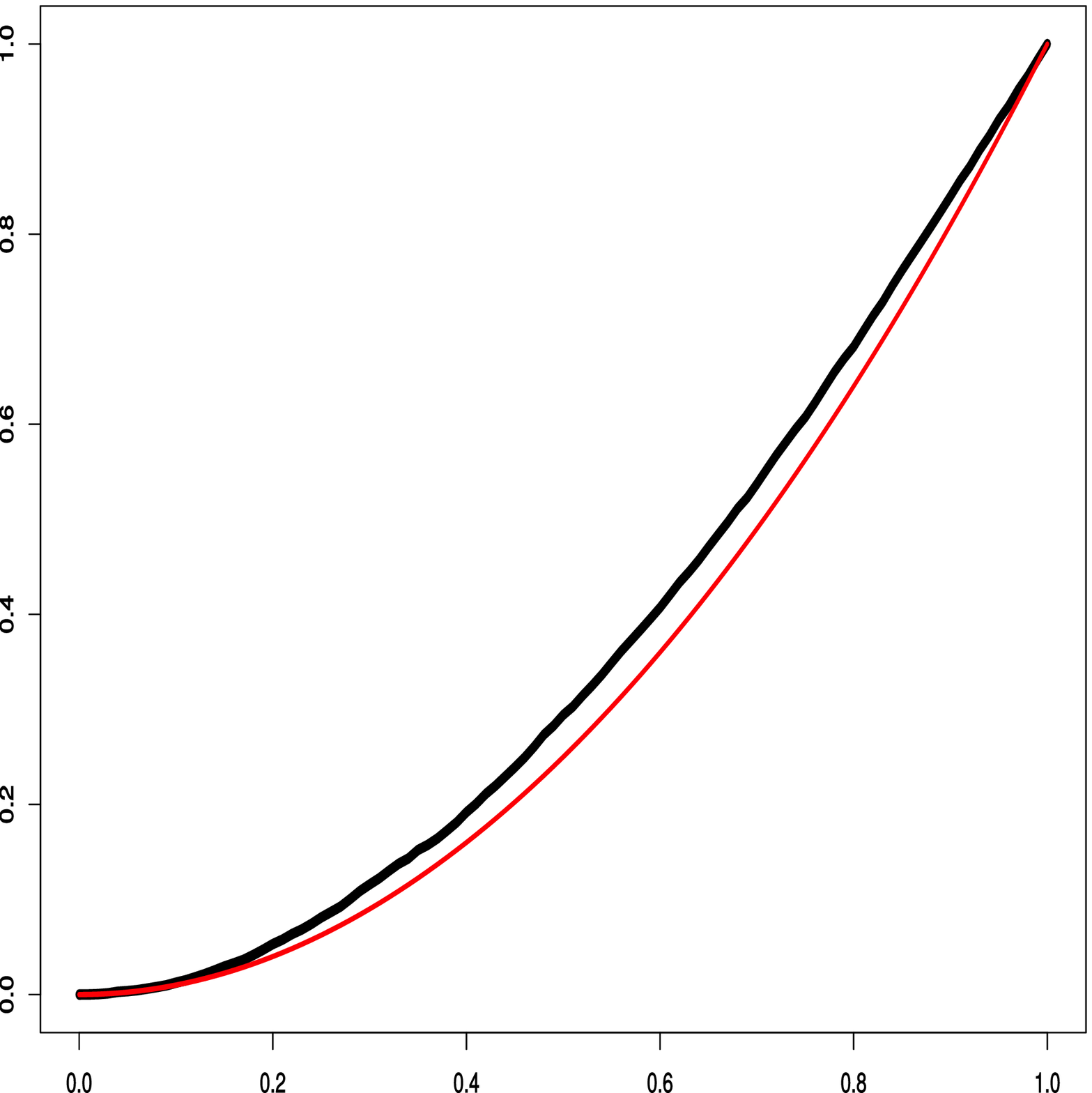}\quad \includegraphics[width=6cm]{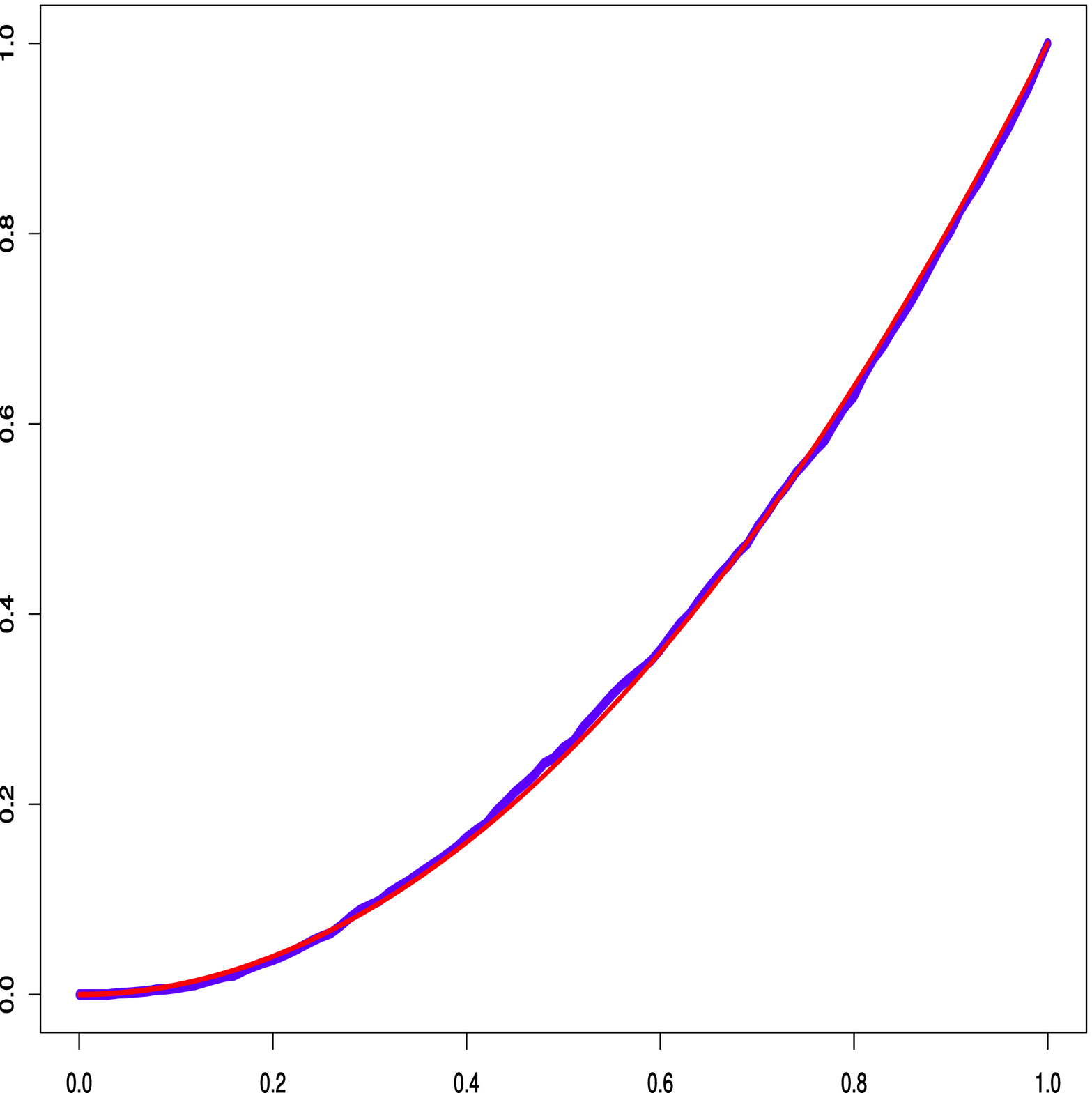}\\
\caption{Simulations illustrating  Theorem \ref{thm:uniformFromAway}. $Y_s$ was simulated 10,000 times for $d=3$ and $s=50$. The  images in the left column are showing the empirical distribution of $Y_s$ in the unit disk and the empirical cumulative distribution function of $|Y_s|$.
The images in the right column are analogous except that they represent $|Y_s|$ conditioned on $\{s-S_{N_s}\geq 3\}$. The red curves in the lower images are graphs of the function $y=x^2$, i.e., the theoretical asymptotic cumulative distribution function in the conditional case, when $s\to \infty$.}\label{simulation}
\end{center}
\end{figure}

\begin{proof}
Recall that $\lambda_{d-1}(A)=0$ if and only if $\tau_{\infty}(A)=0$, for all 
$A\in \mathcal{B}(\D)$. Assume that $A\in \mathcal{B}(\D)$ is such that 
$\lambda_{d-1}(\partial A)=0$. We have 
\begin{align*}
 \P&(Y_s\in A,S_{N_s-1} \leq s - \beta(s))\\
 &=\sum_{k=1}^{\infty}\P(Y_s\in A,S_{k-1} \leq s - \beta(s),N_s=k)\\
&=\sum_{k=1}^{\infty}\P(Y_s\in A,S_k>s, S_{k-1} \leq s - \beta(s), S_{k-2}\leq s,\ldots, S_1\leq s).
\end{align*}
Recall from $(\ref{def:ZUSigma})$ that a light ray started at $(-u,
\sig_1,\ldots,\sig_{d-2},\sig_{d-1})$, where $y_k$'s satisfy $\sig_1^2+\ldots+\sig_{d-2}^2+\sig_{d-1}^2=1$, and
whose direction is goverened by angles $\Theta^k$, $\Phi_1^k,\ldots, \Phi_{d-2}^k$ intersects the plane $\{x: x_1=0\}$ at $(0,\widetilde{Z}^{k}_{u,\sig})$. This and our  construction of the process yield
\begin{align*}
\P&(Y_s\in A,S_{N_s-1} \leq s - \beta(s))\\
&=\sum_{k=1}^{\infty}\P(Z_{s-S^{x_1}_{k-1},S_{k-1}^{x_{2\ldots d}}}^k\in A, S_{k-1}^{x_1} \leq s - \beta(s), S_{k-2}^{x_1}\leq s,\ldots, S_1^{x_1}\leq s)\\
&=\sum_{k=1}^{\infty}\int_{\Sph}\int_{\beta(s)}^{\infty}\P(Z_{u,\sig}^k\in A)\P(S_{k-1}^{x_{2\ldots d}}\in d\sig, s-S_{k-1}^{x_1}\in du , S_{k-2}\leq s,\ldots, S_1\leq s).
\end{align*}
Using Lemma \ref{prop:markovChainRandomWalk} (c) for $T=k-1$, we obtain
\begin{align*}
\P&(Y_s\in A,S_{N_s-1} \leq s - \beta(s))\\
&=\sum_{k=1}^{\infty}
\int_{\Sph}\int_{\beta(s)}^{\infty}
\P(Z_{u,\sig}^k\in A\mid Z_{u,\sig}^k\in \D)\P(Z_{u,\sig}^k\in \D)\P(S_{k-1}^{x_{2\ldots d}}\in d\sig) \\
& \qquad \qquad \times 
\P(s-S_{k-1}^{x_1}\in du, S_{k-2}\leq s,\ldots, S_1\leq s).
\end{align*}
Let $\Sigma$ be a random vector uniformly distributed  on $\Sph$.
By Proposition \ref{prop:unifrom:ratio} and \eqref{eq:sphere:integral:tau}, for $s\to \infty$,
\begin{align*}
\P&(Y_s\in A,S_{N_s-1} \leq s - \beta(s))\\
&\sim \sum_{k=1}^{\infty}\int_{\Sph}\int_{\beta(s)}^{\infty}\tau_{\infty}(U^{\sig *}(A))\P(Z_{u}^k\in \D)\P(\Sigma\in d\sig) \P(s-S_{k-1}^{x_1}\in du, S_{k-2}\leq s,\ldots, S_1\leq s)\\
&=\int_{\Sph}\tau_{\infty}(U^{\sig *}(A))\P(\Sigma\in d\sig)\sum_{k=1}^{\infty}\int_{\beta(s)}^{\infty}\P(Z_{u}^k\in \D)\P(s-S_{k-1}^{x_1}\in du, S_{k-2}\leq s,\ldots, S_1\leq s)\\
&=\lambda_{d-1}(A)/\lambda_{d-1}(\D)\P(S_{N_s-1} \leq s - \beta(s)).
\end{align*}
\end{proof}

\section{Brightness singularity}\label{sec:bright}
It was shown in \cite[Thm.~5.10]{laser} that there is an apparent brightness singularity at the center of the tube opening
when the dimension is $d=3$. We will now  generalize that result to all dimensions $d\geq 3$.

Let $\bv_s = (S_{N_s} - S_{N_s-1})/|S_{N_s} - S_{N_s-1}| $ be the unit
vector representing the direction of the light ray at the exit time.
Let $\cB(r) = \{(x_1,\ldots, x_d): x_1^2 + x_2^2 +\ldots+ x_d^2 = 1, x_2^2+\ldots +x_d^2 \leq r^2 , x>0 \}$ denote a ball on the unit sphere and recall that 
$B_r(0)=\{(x_2,\ldots, x_d)\in \D\ :\ x_2^2+\ldots+x_d^2\leq r^2\}$.

\begin{thm}\label{thm:u20_1}
For any $0 < r_1 < r_2 < 1$,
\begin{align*}
&\lim_{s\to \infty} \lim_{\delta\to 0}
\frac {s^{d-2}} {\delta^{d-1}} \P\left(\bv_s \in \cB\left( \frac{r_2}{s}\right) \setminus \cB\left( \frac{r_1}{s}\right) , Y_s \in B_\delta(0)\right) 
=\frac{\pi^{-\frac{d-3}{2}}}{4\Gamma(\frac{d+1}{2})\E X_1^2}\left(\frac{r_2^{d-2}-r_1^{d-2}}{d-2}\right).
\end{align*}
\end{thm}

 For the proof we will need a modified version of Lemma 5.11 
from \cite{laser}.
Let $M_s(A) $ be the number of $k\leq N_s -1$ such that $S^{x_1}_k-s \in A$, that is
$$M_s(A)= \sum_{k=0}^{\infty}\1(S_k^{x_1}-s\in A, \max\{S_1^{x_1},\ldots, S_k^{x_k}\}\leq 0). $$
We define $\M_s(A)=\E(M_s(A))$. It is clear that $\M_s(A)$ is a non-negative measure.

\begin{lemma}\label{u22.2}
For any $0 < a_1 < a_2 < \infty$, 
\begin{equation}\label{eq:u22.2}
\lim _{s\to \infty} \frac 1 {s^2} \M_s(-s a_2, -s a_1)
= \frac{1}{2\E X_1^2}(a^2_2 - a^2_1).
\end{equation}
Moreover, 
for every continuous function $f$ and a compact subset $K$ of $(-\infty,0)$ we have
\begin{equation}
\label{eq:vagueConv}\lim_{s\to\infty}\frac{1}{s^2}\int_{K}f(u)\M_s(s\, du)=-\frac{1}{\E X_1^2}\int_{K}f(u)u\, du.
\end{equation}
\end{lemma}

\begin{proof}
 The formula $(\ref{eq:u22.2})$ can be proved just like Lemma 5.11 in \cite{laser}.
For \eqref{eq:vagueConv} first note that there exist $A<B<0$ such that $K\subset (A,B)$.
If we restrict the measures to the interval $[A,B]$ then we can apply  \cite[Prop.~7.19]{folland} to show that $\frac{1}{s^2}\M_s(s\, du)$ converges vaguely to $-\frac{1}{\E X_1^2}u\, du$.
This proves \eqref{eq:vagueConv}.
\end{proof}

\begin{lemma}\label{lemma:aproxProbability}
We have
\begin{align*}
 \lim_{s\to\infty}\lim_{\delta\to 0}\frac{s^{d-2}}{\delta^{d-1}} & \P(S^{x_1}_{N_s-1}\in (s(1-\alpha_2)+\beta_2,s(1-\alpha_1)+\beta_1),Y_s\in B_{\delta}(0))
\\
& =\frac{\pi^{-\frac{d-3}{2}}}{4\Gamma(\frac{d+1}{2})\E X_1^2}\left(\frac{\alpha_1^{-d+2}-\alpha_2^{-d+2}}{d-2}\right).
\end{align*}
 
\end{lemma}
\begin{proof}
Recall the definition of $\widetilde{Z}_{u,\sig}$ from \eqref{def:ZUSigma}.
 We have 
\begin{align*}
 & \P(S^{x_1}_{N_s-1}-s\in (-\alpha_2s+\beta_2,-\alpha_1s+\beta_1),Y_s\in B_{\delta}(0))\\
&= \sum_{k=1}^{\infty} \P\Big(\widetilde{Z}_{s-S^{x_1}_{k-1},S^{x_{2,\ldots,d}}_{k-1}}\in B_{\delta}(0),S^{x_1}_{k-1}-s\in (-\alpha_2s+\beta_2,-\alpha_1s+\beta_1),\\
& \qquad \qquad\max\{S_1^{x_1},\ldots,S_{k-1}^{x_1}\}<0\Big).
\end{align*}
Using a similar approach as in the proof of Theorem \ref{thm:uniformFromAway} and the fact that $B_{\delta}(0)$ is an invariant set for orthogonal operators, we obtain
 \begin{align*}
 & \P(S^{x_1}_{N_s-1}-s\in (-\alpha_2s+\beta_2,-\alpha_1s+\beta_1),Y_s\in B_{\delta}(0))\\
&=\int^{-\alpha_1+\beta_1/s}_{-\alpha_2+\beta_2/s}\P(\widetilde{Z}_{-h}\in B_{\delta}(0))\sum_{k=1}^{\infty}\P(S^{x_1}_{k-1}-s\in dh,\max\{S_1^{x_1},\ldots,S_{k-1}^{x_1}\}<0)\\
&=\int^{-\alpha_1+\beta_1/s}_{-\alpha_2+\beta_2/s}\P(\widetilde{Z}_{-su}\in B_{\delta}(0))\M_s(s\, du).
\end{align*}
By Proposition \ref{propo:densityAsymptotics}, 
$$\lim_{s\to\infty}\frac{\P(\widetilde{Z}_{-su}\in B_{\delta}(0))}{\lambda^{d-1}(B_{\delta}(0))}=f_{s(-u)}(0,\ldots,0).$$
Hence, using the fact that $\lambda^{d-1}(B_{\delta}(0))=\frac{\pi^{\frac{d-1}{2}}}{\Gamma(\frac{d+1}{2})}\delta^{d-1}$
and Proposition \ref{propo:densityAsymptotics}, we have
for $s\to \infty$,
\begin{align*}
 & \lim_{\delta\to 0}\frac{s^{d-2}}{\delta^{d-1}}\P(S^{x_1}_{N_s-1}-s\in (-s/r_1+\alpha,-s/r_2+\beta),Y_s\in B_{\delta}(0))\\
&=\frac{\pi^{\frac{d-1}{2}}s^{d-2}}{\Gamma(\frac{d+1}{2})}\int^{-\alpha_1+\beta_1/s}_{-\alpha_2+\beta_2/s}f_{s(-u)}(0,\ldots,0)\M_s(s\, du)\\
&\sim\frac{\pi^{\frac{d-1}{2}}s^{d-2}}{\Gamma(\frac{d+1}{2})}\int^{-\alpha_1+\beta_1/s}_{-\alpha_2+\beta_2/s}\frac{1}{4(s(-u))^d\pi^{d-2}}\M_s(s\, du)\\
&=\frac{\pi^{\frac{d-1}{2}}}{\Gamma(\frac{d+1}{2})}\int^{-\alpha_1+\beta_1/s}_{-\alpha_2+\beta_2/s}\frac{1}{4(-u)^d s^2 \pi^{d-2}}\M_s(s\, du).
\end{align*}
By Lemma \ref{u22.2},
\begin{align*}
 \lim_{s\to \infty}\int^{-\alpha_1+\beta_1/s}_{-\alpha_2+\beta_2/s}\frac{1}{4(-u)^d}\frac{\M_s(s\, du)}{s^2}
&=\int^{-\alpha_1}_{-\alpha_2}\frac{1}{4(-u)^d}\cdot \frac{-u}{\E X_1^2}\, du\\
&= \frac{1}{4\E X_1^2}\left(\frac{\alpha_1^{-d+2}-\alpha_2^{-d+2}}{d-2}\right).
\end{align*}
The lemma follows from the last two displayed formulas.
\end{proof}

\begin{proof}[Proof of Theorem \ref{thm:u20_1}]

It is elementary to check that if $y = (y_1, y_2, \dots, y_d)= (y_1, y_{2,\ldots, d}) \in \R^d$, 
$\|y\|\leq 1$, $0<\delta<1$, and $x\leq s$ then the following conditions are equivalent,
\begin{align}\label{j22.1}
& \frac{(s,\delta y_{2,\ldots, d})-(x,0,\ldots, 0,-1)}
{\|(s,\delta y_{2,\ldots, d})-(x,0,\ldots, 0,-1)\|}
\in \cB\left( \frac{r_2}{s}\right) \setminus \cB\left( \frac{r_1}{s}\right), \\
&x\in \left[s-\|\delta y_{2,\ldots, d}-(0,\ldots, 0,-1)\|\sqrt{\frac{s^2}{r_2^2}-1},s-\|\delta y_{2,\ldots, d}-(0,\ldots, 0,-1)\|\sqrt{\frac{s^2}{r_1^2}-1}\right]. \label{j22.2}
\end{align}

By Lemma \ref{prop:markovChainRandomWalk} the random vector $S_{N_s-1}^{x_{2\ldots d}}$ is uniformly distributed on $\Sph$ and is independent of the process 
$(S_{k}^{x_1})$. Therefore, we have  
\begin{align}\label{j22.3}
&\P\left(\bv_s \in \cB\left( \frac{r_2}{s}\right) \setminus \cB\left( \frac{r_1}{s}\right) , Y_s \in B_\delta(0)\right)\\
&=\P\left(\bv_s \in \cB\left( \frac{r_2}{s}\right) \setminus \cB\left( \frac{r_1}{s}\right) , Y_s \in B_\delta(0)\, |\, S_{N_s-1}^{x_{2\ldots d}}=(0,\ldots,0,-1)\right)\nonumber\\
&=\P\Bigg(\frac{(s,Y_s)-(S_{N_s-1},0,\ldots,-1)}{\|(s-S_{N_s-1}^{x_1},Y_s-(0,\ldots,0,-1))\|} \in \cB\left( \frac{r_2}{s}\right) \setminus \cB\left( \frac{r_1}{s}\right) ,\nonumber \\
&\qquad\qquad Y_s \in B_\delta(0)\, |\, S_{N_s-1}^{x_{2\ldots d}}=(0,\ldots,0,-1)
\Bigg).\nonumber
\end{align}
We use the substitution $x=S_{N_s-1}$ and $y_{2,\ldots, d}=Y_s/\delta$ in \eqref{j22.1}-\eqref{j22.2} to see that 
\begin{align}\label{j22.4}
A_s^{\delta}:=&{\textstyle \left\{\frac{(s,Y_s)-(S_{N_s-1},0,\ldots,-1)}{\|(s-S_{N_s-1}^{x_1},Y_s-(0,\ldots,0,-1))\|} \in \cB\left( \frac{r_2}{s}\right) \setminus \cB\left( \frac{r_1}{s}\right), Y_s \in B_\delta(0)\right\}}\\
=&{\textstyle \left\{S_{N_s-2}\in \left[s-\|Y_s-(0,\ldots, 0,-1)\|\sqrt{\frac{s^2}{r_2^2}-1},\right.\right.} \nonumber\\
&\qquad {\textstyle \left.\left. s-\| Y_s-(0,\ldots, 0,-1)\|\sqrt{\frac{s^2}{r_1^2}-1}\right], Y_s\in B_{\delta}(0)\right\}.}\nonumber
\end{align}
If $Y_s\in B_{\delta}(0)$ and $\delta \to 0$ then
 $\|Y_s-(0,\ldots, 0,-1)\|\to 1$. Hence, for large $s$ and small $\delta$, 
$$s-\frac{s}{r_j}-2\leq s-\|Y_s-(0,\ldots, 0,1)\|\sqrt{\frac{s^2}{r_j^2}-1}\leq s-\frac{s}{r_j}+2.$$
Let
$$A_s^{\delta\pm}=\left(S_{N_s-1}\in \left[s\left(1-\frac{1}{r_1}\right)\mp 2,s\left(1-\frac{1}{r_2}\right)\pm 2\right],Y_s\in B_{\delta}(0)\right),$$
and note that  $A^{\delta+}_s\supset A^{\delta}_s \supset A^{\delta-}_s$
for large $s$ and small $\delta$.
By Lemma \ref{lemma:aproxProbability} 
$$\lim_{s\to\infty}\lim_{\delta\to 0}\frac{s^{d-2}}{\delta^{d-1}}\P(A^{\delta+}_s)=\lim_{s\to\infty}\lim_{\delta\to 0}\frac{s^{d-2}}{\delta^{d-1}}\P(A^{\delta-}_s)
=\frac{\pi^{-\frac{d-3}{2}}}{4\Gamma(\frac{d+1}{2})\E X_1^2}\left(\frac{r_2^{d-2}-r_1^{d-2}}{d-2}\right).$$
Combining this with \eqref{j22.3}-\eqref{j22.4} proves the theorem.
\end{proof}

\section*{Acknowledgments}
The authors would like to thank Sara Billey for very helpful advice.
The second author is grateful to Microsoft Corporation for the allowance on Azure where the simulation illustrated in Figure \ref{simulation} was performed. We are grateful to the anonymous referee for many suggestions for improvement.

\bibliographystyle{alpha}

\bibliography{refs_laser1}

\begin{thebibliography}{CPSV10b}

\bibitem[ABS13]{ABS}
Omer Angel, Krzysztof Burdzy, and Scott Sheffield.
\newblock Deterministic approximations of random reflectors.
\newblock {\em Trans. Amer. Math. Soc.}, 365(12):6367--6383, 2013.

\bibitem[AG80]{arcsine}
Barry~C. Arnold and Richard~A. Groeneveld.
\newblock Some properties of the arcsine distribution.
\newblock {\em J. Amer. Statist. Assoc.}, 75(369):173--175, 1980.

\bibitem[BGT87]{regularVariation}
N.~H. Bingham, C.~M. Goldie, and J.~L. Teugels.
\newblock {\em Regular variation}, volume~27 of {\em Encyclopedia of
  Mathematics and its Applications}.
\newblock Cambridge University Press, Cambridge, 1987.

\bibitem[BT16]{laser}
Krzysztof Burdzy and Tvrtko Tadi\'c.
\newblock Can one make a laser out of cardboard?
\newblock 2016.
\newblock to appear in {\it Ann. Appl. Probab.}, Arxiv:1507.00961.

\bibitem[CPSV09]{comets1}
Francis Comets, Serguei Popov, Gunter~M. Sch{\"u}tz, and Marina Vachkovskaia.
\newblock Billiards in a general domain with random reflections.
\newblock {\em Arch. Ration. Mech. Anal.}, 191(3):497--537, 2009.

\bibitem[CPSV10a]{comets2}
Francis Comets, Serguei Popov, Gunter~M. Sch{\"u}tz, and Marina Vachkovskaia.
\newblock Knudsen gas in a finite random tube: transport diffusion and first
  passage properties.
\newblock {\em J. Stat. Phys.}, 140(5):948--984, 2010.

\bibitem[CPSV10b]{comets3}
Francis Comets, Serguei Popov, Gunter~M. Sch{\"u}tz, and Marina Vachkovskaia.
\newblock Quenched invariance principle for the {K}nudsen stochastic billiard
  in a random tube.
\newblock {\em Ann. Probab.}, 38(3):1019--1061, 2010.

\bibitem[Don80]{doney}
R.~A. Doney.
\newblock Moments of ladder heights in random walks.
\newblock {\em J. Appl. Probab.}, 17(1):248--252, 1980.

\bibitem[Eva01]{evans}
Steven~N. Evans.
\newblock Stochastic billiards on general tables.
\newblock {\em Ann. Appl. Probab.}, 11(2):419--437, 2001.

\bibitem[Fol99]{folland}
Gerald~B. Folland.
\newblock {\em Real analysis}.
\newblock Pure and Applied Mathematics (New York). John Wiley \& Sons, Inc.,
  New York, second edition, 1999.
\newblock Modern techniques and their applications, A Wiley-Interscience
  Publication.

\bibitem[Ver77]{veraveb}
N.~Veraverbeke.
\newblock Asymptotic behaviour of {W}iener-{H}opf factors of a random walk.
\newblock {\em Stochastic Processes Appl.}, 5(1):27--37, 1977.

\end{thebibliography}

\end{document}